\documentclass[11pt]{article}
\usepackage{authblk}
\usepackage[all]{xy}
\usepackage{verbatim}
\usepackage{amssymb,latexsym,amsmath,graphicx,amsfonts,amsthm,eucal,mathrsfs,fullpage, epsfig}
\usepackage{amscd,subfigure,color,xcolor,setspace,cite,mathtools}
\usepackage[font=footnotesize]{caption}
\usepackage{float}
\DeclareMathOperator{\sign}{sign}
\usepackage{caption}
\usepackage{epstopdf}
\usepackage[english]{babel}
\usepackage{babel}
\usepackage{fontenc}
\usepackage[colorlinks=true,citecolor=blue]{hyperref}
\usepackage{float}
\usepackage{lineno}
\usepackage{pgf,tikz}
\usepackage{array}
\usepackage{multirow}
\usepackage{float}
\restylefloat{table}

\usetikzlibrary{arrows}
\usetikzlibrary[patterns]
\usetikzlibrary{decorations.pathmorphing}
 \usetikzlibrary{plotmarks}
        \usetikzlibrary{intersections}
        \usetikzlibrary{shadings}
        \usetikzlibrary{calc}
\newlength{\hatchspread}
\newlength{\hatchthickness}
\newlength{\hatchshift}
\newcommand{\hatchcolor}{}
\tikzset{hatchspread/.code={\setlength{\hatchspread}{#1}},
         hatchthickness/.code={\setlength{\hatchthickness}{#1}},
         hatchshift/.code={\setlength{\hatchshift}{#1}},
         hatchcolor/.code={\renewcommand{\hatchcolor}{#1}}}
\tikzset{hatchspread=3pt,
         hatchthickness=0.4pt,
         hatchshift=1pt,
         hatchcolor=black}
\pgfdeclarepatternformonly[\hatchspread,\hatchthickness,\hatchshift,\hatchcolor]
   {custom north west lines}
   {\pgfqpoint{\dimexpr-2\hatchthickness}{\dimexpr-2\hatchthickness}}
   {\pgfqpoint{\dimexpr\hatchspread+2\hatchthickness}{\dimexpr\hatchspread+2\hatchthickness}}
   {\pgfqpoint{\dimexpr\hatchspread}{\dimexpr\hatchspread}}
   {
    \pgfsetlinewidth{\hatchthickness}
    \pgfpathmoveto{\pgfqpoint{0pt}{\dimexpr\hatchspread+\hatchshift}}
    \pgfpathlineto{\pgfqpoint{\dimexpr\hatchspread+0.15pt+\hatchshift}{-0.15pt}}
    \ifdim \hatchshift > 0pt
      \pgfpathmoveto{\pgfqpoint{0pt}{\hatchshift}}
      \pgfpathlineto{\pgfqpoint{\dimexpr0.15pt+\hatchshift}{-0.15pt}}
    \fi
    \pgfsetstrokecolor{\hatchcolor}
    \pgfusepath{stroke}
   }
\usetikzlibrary{arrows}
\usetikzlibrary[patterns]
\numberwithin{equation}{section}
\newtheorem{theorem}{Theorem}[section]
\newtheorem{definition}[theorem]{Definition}
\newtheorem{lemma}[theorem]{Lemma}
\newtheorem{example}[theorem]{Example}
\newtheorem{proposition}[theorem]{Proposition}
\newtheorem{corollary}[theorem]{Corollary}
\newtheorem{remark}{Remark}[section]

\def\bbr{{\mathbb R}}

\def\bbc{{\mathbb C}}
\def\bbd{{\mathbb D}}

\def\la#1{\hbox to #1pc{\leftarrowfill}}
\def\ra#1{\hbox to #1pc{\rightarrowfill}}
\def\fract#1#2{\raise2pt\hbox{$ #1 \atop #2 $}}
\def\ds{\displaystyle}
\def\x{{\bf x}}
\def\ds{\displaystyle}

\AtBeginDocument{\addtocontents{toc}{\protect\setlength{\parskip}{0pt}}}
\begin{document}\title{\textbf{Constrained polynomial roots and a modulated approach to  Schur stability}}
\author[1, 2]{ Ziyad AlSharawi\thanks{Corresponding author: zsharawi@aus.edu. This work was done while the first author was on sabbatical leave from the American University of Sharjah.}}
\author[2]{ Jose S. C\'anovas}
\author[1]{ Sadok Kallel}

\affil[1]{\small American University of Sharjah, P. O. Box 26666, University City, Sharjah, UAE}
\affil[2]{\small Universidad Politécnica de Cartagena, Paseo de Alfonso XIII 30203, Cartagena, Murcia, Spain}
\date{\today}
\maketitle

\begin{abstract}
It is common in stability analysis to linearize a system and investigate the spectrum of the Jacobian matrix. This approach faces the challenge of determining the matrix spectrum when the coefficients depend on parameters or when the characteristic polynomial is more than quartic. In this paper, we reverse the classical process and use the authors'  work on global stability to find sufficient conditions on the coefficients that ensure the zeros of the characteristic polynomial are in the open unit disk. This leads to an algorithm that begins by testing the $\ell_1$-norm of the polynomial, and if it is not less than two, perform an iteration process that can be implemented with moderate effort. We give examples that show the effectiveness of our method when compared with the Jury's algorithm. Last, we formalize our constructions in terms of semialgebraic sets.

\end{abstract}

\section{ Introduction}
A square matrix is called Schur-stable if its spectrum is contained within the open unit disk $\mathbb{D}$. Since a characteristic polynomial determines the spectrum of a square matrix, we say a polynomial is Schur-stable if all of its zeros are located inside  $\mathbb{D}$. Spectra of square matrices and roots of polynomials have significant implications across various disciplines such as science, engineering, and economics \cite{Ho-Jo2013,El1999,Mu1989,Br-Ch2012,Og2015,Mi2020,Ga2009}, and the substantial plethora of research and pertinent literature concerning this topic is becoming challenging to trace.  In continuous dynamical systems, it is essential to determine whether the spectrum of a matrix lies within the left-half plane. On the other hand, discrete dynamical systems are mainly concerned with determining whether the spectrum falls within  $\mathbb{D}$. By applying a M\"{o}bius transform to polynomials, it can be shown that both are two sides of the same coin and, as a result, addressing a query in one area has a parallel analog in the other. Therefore, this paper will concentrate on determining whether the zeros of a real polynomial lie within  $\mathbb{D}$. The search for an effective method or algorithm for checking the Schur stability of a real-valued polynomial has reached a satisfactory milestone. However, an algorithm's ease of implementation and simplicity are the significant factors that spur more creative ideas in this direction, and we aim to provide a convenient approach that contributes toward this goal.
\\

Consider the $n^{th}$ degree monic-polynomial
  \begin{equation}\label{polynomial}
  p(z)=z^n+a_{n-1}z^{n-1}+a_{n-2}z^{n-2}+\cdots+a_1z+a_0,
  \end{equation}
where the coefficients are real numbers. The main problem of ``root clustering'' is to find conditions on the coefficients so that the roots lie in a predefined domain of $\bbc$. This is a vast area of research that cannot be summarized here. We can, however, formulate this problem qualitatively and in complete generality, as we do in \S\ref{algebraictheory},  and show that if the roots are to lie in a semialgebraic subset $B$ of the plane $\bbr^2$, then the coefficient locus, written $\mathcal C_n(B)$, must be a semialgebraic subset of $\bbr^n$. This is a framework that goes back to the work of Kalman (see \cite{lasserre}).
\\

In this paper, we are concerned with the classical Schur problem of finding conditions on the roots of the real polynomial $p$ in Eq. \eqref{polynomial} so that its roots are contained strictly in the unit disk. We offer a new method, fast and simple, to decide about this problem. 
The starting point is the well-known result that $p$ is Schur-stable if its $\ell_1$-norm $\|p\|_1=1+\sum_{j=1}^{n-1}|a_j|$ is less than $2$ (\cite{smithies}, Theorem I). This is derived by use of 
Rouch\'e's theorem in complex analysis (see Proposition 1.3.1 in \cite{ladas}) or by our own analysis-free proof given in Proposition \ref{Pr-MainTheorem} here.  We call this the ``$\ell_1$-condition''. We then develop an algorithm that stems from global stability in maps of mixed monotonicity. This starts with the $\ell_1$ condition on the coefficients of a polynomial and then iterates a linear difference equation to produce a sequence of conditions, each of which is sufficient to ensure that the roots of the polynomial are within the open unit disk. If the $\ell_1$-condition fails for an iterated system, one tries the next, and so forth. The method is inconclusive if one cannot establish the $\ell_1$-condition for any finite number of iterates. The advantage of this approach is the ease of its implementation, as illustrated through some examples that show when the method works, it works at a minimal cost and effort, as opposed to Jury's classical algorithm \cite{Ju1963}, which is generally cumbersome when parameters are involved or when the polynomial degree is high \cite{Ga-Sc-Su-Tr-We2021}. \\

The main idea is to tackle the Schur stability problem from a global stability perspective in discrete-time dynamical systems. Based on the authors' recent work on global stability \cite{Al-Ca-Ka2023}, we connect the polynomial $p$ in Eq. \eqref{polynomial} to the $n^{th}$-order linear difference equation
\begin{equation}\label{Eq-nthDegreeDE}
\begin{split}
  x_{k+1}=&F(x_k,x_{k-1},\ldots,x_{k-n+1})\\
  =&-a_{n-1}x_k-a_{n-2}x_{k-1}-\cdots-a_0x_{k-n+1}.
  \end{split}
  \end{equation}
The delay of the linear difference equation is the degree of the polynomial, and the eigenvalues of the Jacobian matrix of the system are the zeros of the polynomial. The basic theory of linear dynamical systems asserts that $0$ is a global attractor for Eq. \eqref{Eq-nthDegreeDE}
if and only if the roots of $p$ in Eq. \eqref{polynomial} are in the unit disk. The Schur problem becomes, therefore, equivalent to giving conditions on the global stability of the zero equilibrium of Eq. \eqref{Eq-nthDegreeDE}. Since local and global stability are equivalent for linear systems, it turns out that our techniques in \cite{Al-Ca-Ka2023} produce the needed constraint on the coefficients to ensure global stability. If this constraint is not applicable, we substitute a linear difference equation in one of the variable entries to obtain a new system with greater delay (which we call an iterate system) and, thus, a new polynomial. Then, we apply our technique again to this new system. The global stability of any iterate system implies global stability for the original system, so if stability occurs at any finite stage, the algorithm stops, and we answer the original question positively.
In so doing, we obtain a sequence of \textit{sufficient} conditions starting with the $\ell_1$ condition. \\

The algorithm introduced in this paper is ``incremental'' as opposed to being an ``all or nothing'' algorithm, as in the case of Jury's algorithm. It is also ``systematic'' instead of ``ad-hoc'' as with many techniques used to address the examples discussed in this paper. Geometrically, as we point out, the algorithm is producing a \textit{semialgebraic filtration}, which is a nested increasing sequence of semialgebraic sets of increasing degrees, that is progressively filling out the coefficient locus $\mathcal C_n(\mathbb{D})$ in $\mathbb R^n$ of all real polynomials whose roots are in the unit disk. If the filtration fills the whole locus, then we obtain \textit{necessary} conditions as well, but this is unresolved in this paper and remains under investigation.  
\\

The organization of this paper is as follows: In Section 2, we streamline the general relationship between roots and coefficients loci by using the language of semialgebraic sets. Section three provides a concise overview of the embedding technique and the global stability result from \cite{Al-Ca-Ka2023}, which is essential for establishing our algorithm. Section four focuses on an illustrative and simplified scenario involving a second-degree polynomial. We illustrate the practical application of the theory in this specific example and, subsequently, extend the algorithm to polynomials of a higher degree. In section five, we present and prove our algorithm stating sufficient conditions based on the polynomial coefficients. Section six provides concrete examples and practical applications that illustrate the advantages and limitations of our technique.  In section seven, we restate our algorithm geometrically as an increasing semialgebraic filtration of increasing degree.  Finally, we close with a summary of the main results of this paper.


\section{The algebraic theory}\label{algebraictheory}

Determining conditions on the coefficients of a real polynomial so that the roots are in predefined subsets of $\mathbb C$ is part of the general theory of \textit{root-clustering}. In this section, we state and prove a general result in that regard. Although our Schur stability algorithm in \S\ref{algorithm} does not depend on the results of this section, the results here give a well-rounded context. \\

We first recall that a \textit{multivariate polynomial} in $n$ variables  $x_1,\ldots , x_n$ is a finite sum of monomials of the form $x_{i_1}^{j_1}\cdots x_{i_r}^{j_r}$, $1\leq i_j\leq n$, $j_t\in\mathbb N$, with coefficients in $\bbr$. The collection of all such polynomials forms a ring commonly written
$\bbr [x_1,\ldots, x_n]$. A semialgebraic subset of $\bbr^n$ is any subset defined in terms of polynomial equalities and inequalities. More precisely, the semialgebraic sets form the smallest Boolean collection generated by subsets of the form $\{\x = (x_1,\ldots, x_n)\in \bbr^n\ |\  P(\x ) > 0\}$, where $P(\x ) := P(x_1, ..., x_n)$ is a multivariate polynomial. In other words, this is the collection obtained by finite intersections, finite unions, and complements of these subsets given by polynomial inequalities\cite{Co2002}. This is also equivalent to the following formulation: a subset $X\subset\bbr^n$ is semialgebraic if it is the union of finitely many subsets (the ``basic semialgebraic sets'') of the form
\begin{equation}\label{semialgebraicset}
\{\x \in\bbr^n\ |\ P(\x ) = 0, g_1(\x )>0, \ldots, g_k(\x )>0\},
\end{equation}
where $P,g_1,\ldots, g_k\in \bbr[x_1,\ldots, x_n]$. The mathematical area of ``real algebraic geometry'' is precisely the study of semialgebraic sets, their properties, and applications. 

Before we give the main result of this section, we say by convention that $B\subset \bbc^n$ is semialgebraic if $B$ is semialgebraic in $\bbr^{2n}$, after identifying $\bbc^n$ with $\bbr^{2n}$ in the standard way. Now, we give the theorem, which is a conceptual refinement of \cite{lasserre}.

\begin{theorem}\label{semialgebraic}
Let $B$ be a semialgebraic set of $\bbc$ and consider the set $\mathcal C_n(B)\subset\bbr^n$ of all coefficients
$(a_0,\ldots, a_{n-1})\in\bbr^n$ of real monic polynomials $p(z) = z^n + a_{n-1}z^{n-1}+\cdots +a_1z+a_0$ of degree $n\geq 1$ whose roots belong to $B$.
Then $\mathcal C_n(B)$ is a semialgebraic set of $\bbr^n$.
\end{theorem}

The set $\mathcal C_n(B)$ is referred to as the \textit{coefficient locus}.
Before we give the proof of this result, we illustrate it with some classical examples.

\begin{example}\rm (Hurwitz problem). Define
$g: \bbc\longrightarrow\bbr, z\mapsto Re(z)$.
The set of polynomials represented by Eq. \eqref{polynomial} whose roots $z$ are in the half-plane $\{z\ |\ g(z) < 0\}$ has been characterized entirely by Hurwitz (these are called Hurwitz polynomials \cite{garzaetal}). In addition, if we want the polynomials $p(z)$ to have real coefficients, then we need the roots to satisfy some further semialgebraic conditions, which we spell out in the proof of the theorem.
\end{example}

\begin{example}\label{schurproblem}\rm  (Schur problem).  Define
$g: \bbc\longrightarrow \bbr, z\mapsto z\bar z$.
The set of real polynomials represented by Eq. \eqref{polynomial} whose roots $z$ are in the unit disk $B=\{z\ |\ g(z) < 1\}$  has been completely characterized by Schur-Cohn and Jury (these are the \textit{Schur polynomials}). Notice that
$B= \mathbb{D}=\{(x,y)\ |\ x^2+y^2<1\}$ is indeed a semialgebraic set of $\bbr^2$. When $n=2$, one has the following description for the image set in the coefficient domain (\cite{Ju1963}, page 146):
\begin{equation}\label{c2B}
\mathcal C_2 (\mathbb{D})= \{(a_0,a_1)\in\bbr^2\ :\ a_0<1,\ a_0+a_1>-1\ ,\ a_0-a_1>-1\}.
\end{equation}
This description of the coefficient locus as a semialgebraic set is not unique, and various authors have equivalent descriptions (Schur-Cohn, Hermite, Samuelson, Farebrother). The case when $n=3$ is discussed in detail in \cite{koji} where it is verified that the Schur stability conditions are equivalent to the Samuelson necessary and sufficient conditions, as simplified by Farebrother below:
$$\mathcal C_3 (\mathbb{D}) = \{(a_0,a_1,a_2)\in\bbr^3\ :\ -1<a_0<1,\ a_0^2-1<a_0a_2-a_1,\ |a_0+a_2|<1+a_1\}.$$
As asserted, $\mathcal C_2 (\mathbb{D})$ and $\mathcal C_3 (\mathbb{D}) $ are semialgebraic sets in $\bbr^2$ and $\bbr^3$, respectively. Notice that the algebraic conditions for $\mathcal C_2(B)$ are linear, while they are both quadratic and linear in the case of $\mathcal C_3(B)$. This is not coincidental, as discussed in \S\ref{stratification}.
\end{example}

Before we embark on the proof of Theorem \ref{semialgebraic}, we first recall the elementary symmetric polynomials on $n$ entries
$$e_k(x_1,\ldots, x_n) = (-1)^k\sum_{1\leq j_1<\cdots <j_k\leq n}x_{j_1}\cdots x_{j_k}\ ,\ 1\leq k\leq n.$$
In particular, $e_1(x_1,\ldots, x_n)=-\sum x_i$ and $e_n(x_1,\ldots, x_n)=(-1)^nx_1x_2\cdots x_n$. These polynomials are symmetric because they satisfy
$e_k(x_{\sigma (1)},\ldots, x_{\sigma (n)}) = e_k(x_1,\ldots, x_n)$ for any permutation $\sigma\in\mathfrak S_n$, where $\mathfrak S_n$ is the symmetric group on $n$-letters.
They are called \textit{elementary} because they generate all symmetric polynomials in $\bbr [x_1,\ldots, x_n]$ (these also form a polynomial ring). Their relevance lies in the fact that the coefficients of a monic polynomial are  precisely the symmetric polynomials evaluated on roots.

Let $\hbox{SP}^{n}(\bbc)$ be the set of \textit{unordered} $n$-tuples of points in $\bbc$, which is obtained as the quotient of $\bbc^n$ by the permutation action of $\mathfrak S_n$. The quotient map is written $\pi : \bbc^n\rightarrow\hbox{SP}^n(\bbc)$. An element in $\hbox{SP}^{n}(\bbc)$ is written as $[z_1,\ldots, z_n]$, $z_i\in\bbc$, with the possibility of repetitions among entries. By construction, $\pi (x_1,\ldots, x_n) = [x_1,\ldots, x_n]$.

A well-known useful result in the field states that the map
   \begin{eqnarray}
   \Psi: \hbox{SP}^{n}(\bbc )&\longrightarrow& \bbc^n\label{mappsi}\\
   \ [z_1,\ldots, z_n]&\longmapsto&(e_1(z_1,\ldots,z_n),\ldots, e_n(z_1,\ldots, z_n))\nonumber
   \end{eqnarray}
is a bijection (in fact, a diffeomorphism). It gives a bijection between the set of roots (the left-hand side) and the set of coefficients given by the elementary polynomials evaluated on the roots (the right-hand side). We call the left-hand side the ``root domain'' and its image on the right the ``coefficient domain''. If the root domain changes, its coefficient domain will change via a bijection. We stress that roots are naturally unordered while coefficients are ordered.

Now, consider the composite
$$
f: \bbc^n\fract{\pi}{\ra 2}\hbox{SP}^n(\bbc)\fract{\Psi}{\ra 2}\bbc^n,
$$
and assume the polynomial 
 $p(z) = z^n+a_{n-1}z^{n-1}+ \cdots + a_1z+a_0$ has roots in $B$, where $B$ is semialgebraic in $\bbc$. The roots are unordered, and there are $n$ of them, so they lie in the set $\pi (B^n)=\hbox{SP}^{n}(B)\subset\hbox{SP}^{n}(\bbc)$.
Moreover, we require that our polynomials have real coefficients, which puts an extra condition on the roots as either real or conjugate pairs.

\begin{lemma}
    The set $\mathcal R_n(B)$ of all tuples $(z_1,\ldots, z_n)\in B^n$ such that $(z-z_1)\cdots (z-z_n)\in \bbr [x_1,\ldots, x_n]$  is a semialgebraic subset of $\bbc^n$.
\end{lemma}

\begin{proof} It is convenient to think of an element $(z_1,\ldots, z_n)$ of $\mathcal R_n(B)$ as ordered roots of a polynomial with real coefficients (first condition), such that each root is in $B\subset\bbc$ (second condition).
Consider the following subset $\mathcal Z_n^P$ of $\bbc^n$, which is defined for a given ``tuple partition" $P=((i_1,\ldots, i_m), (r_1,\ldots, r_\ell), (s_1,\ldots, s_\ell)),$ by the following:
$$\mathcal Z_n^P = \left\{(z_1,\ldots, z_n)\in\bbc^{n} \ ,\ Im(z_{i_1})=\cdots = Im(z_{i_m})=0,\ z_{r_t} = \bar z_{s_t},\ 1\leq t\leq \ell\right\},$$
\begin{equation}\label{mlstratum}
\ \ \hbox{$P$ as above with}\ \{i_1,\ldots, i_m, r_1,\ldots, r_\ell, s_1,\ldots, s_\ell\} = \{1,\ldots, n\},\ m+2\ell =n.
\end{equation}
This is the set of ordered roots, $m$ of which are real and $\ell$ of which have imaginary parts, $m+2\ell =n$. This set can be rewritten as
$$
    \mathcal Z_n^P = \left\{(z_1,\ldots, z_n) \ ,\ \ \ z_{i_j}-\bar{z}_{i_j}=0, 1\leq j\leq m\right\}\ \bigcap\
\left\{(z_1,\ldots, z_n)\ |\ \ z_{r_t} = \bar{z}_{s_t},\ 1\leq t\leq \ell\right\}
$$
so it is, in fact, algebraic.
Define $\mathcal Z_n = \bigcup_P\mathcal Z_n^P$ (union over all such tuple partitions).
By construction, $\mathcal Z_n$ is semialgebraic as a finite union of semialgebraic sets. By construction, we have the identifications
\begin{eqnarray*}
    \mathcal R_n(B)&=&\{(z_1,\ldots, z_n)\in\bbc^n\ |\
    z_i\in B,\ e_k(z_1,\ldots, z_n)\in\bbr, \forall \ 1\leq k\leq n\}\\
    &=&B^n\cap\mathcal Z_n\subset\bbc^n.
\end{eqnarray*}
Since products of semialgebraic sets are semialgebraic, $B^n$ is semialgebraic in $\bbc^n$ (i.e., in $\bbr^{2n}$ as already explained), and its intersection with other semi-algebraic sets is still semialgebraic, so $\mathcal R_n(B)$ is semialgebraic.    
\end{proof}
Now, we are in a position to give the proof of Theorem \ref{semialgebraic}.
\begin{proof} (of Theorem \ref{semialgebraic}) Observe that the root locus in $\hbox{SP}^n(\bbc)$ of all real polynomials with roots in $B$, is by construction the image under $\pi$ of $\mathcal R_n(B)$, and the coefficient locus is the image under $\Psi\circ\pi$ of $\mathcal R_n(B)$. Note that this image, which we denoted by $\mathcal C_n(B)$, lies in $\bbr^n$, where $\bbr^n\subset\bbc^n$ corresponds to the subspace of real parts. The claim of the Theorem is that this image is semi-algebraic.  We can represent all maps within a convenient commuting diagram
$$\xymatrix{
\mathbb{C}^n\supset R_n(B)\ar[d]^\pi\ar[rr]^-{\phi}&&
\mathcal C_n(B)\subset \mathbb{R}^n\ar[d]^\subset\\
\hbox{SP}^n(C)\ar[rr]^{\Psi}&&\mathbb{C}^n}
$$
Since the map $\phi = \Psi\circ\pi$ is obtained from the elementary symmetric functions, it is a polynomial map; in other words, this is a map when written from $\bbr^{2n}=\bbc^n$ into $\bbr^n$ has the form
$$\phi (x_1,\ldots, x_{2n}) = (p_1(x_1,\ldots, x_{2n}),\ldots, p_n(x_1,\ldots, x_{2n}))\ \ \ ,\ \ \ p_i\in\bbr[x_1,\ldots, x_{2n}]$$
(The case $n=2$ is treated in detail in Example \ref{Ex-Sadok}).
We can then appeal to a theorem from the theory of semialgebraic sets, which states that the image by a polynomial map $\bbr^n\rightarrow\bbr^m$  of a semialgebraic set in $\bbr^n$ is also semialgebraic in $\bbr^m$ (\cite{coste}, Corollary 2.4). Apply this theorem to $\phi$, which is the restriction of $\Psi\circ\pi$ to the root locus $\mathcal R_n(B)$ in $\bbc^n$, so that its image
$\mathcal C_n(B)$ is semialgebraic as desired.
\end{proof}
We close this section with an example that illustrates the above theory and analyzes in detail Example \ref{schurproblem}, corresponding to when roots are inside the unit disk $\bbd$.

\begin{example}\label{Ex-Sadok}\rm
 We consider the case of degree two polynomials
$p(z) = z^2+a_1z+a_0$ and explain how Theorem \ref{semialgebraic} works. Here
$$\mathcal R_2(\bbd ) = \{(z_1,z_2)\ :\ |z_1|<1, |z_2|<1, z_1+z_2\in\bbr, z_1z_2\in\bbr\} $$
and the map $\Psi\circ\pi (z_1,z_2) = (-z_1-z_2,z_1z_2)$.
We can rewrite this root locus as follows:
\begin{eqnarray*}
    \mathcal R_2(\bbd ) &=&
\{(z_1,z_2)\ :\ |z_1|<1, |z_2|<1\}\bigcap
\left(
\{(z_1,z_2)\ :\ z_1,z_2\in\bbr\}\cup
\{(z,\bar z),\; z\in \mathbb{C}\}
\right).
\end{eqnarray*}
To describe its image locus $\mathcal C_2(\bbd )$ in $\bbr^2$, write $z_1=\alpha_1 + \beta_1i$, and $z_2=\alpha_2+\beta_2i$, and identify $(z_1,z_2)$ with $(\alpha_1,\beta_1,\alpha_2,\beta_2)$ in $\bbr^4$. Then
$$\Psi\circ\pi (\alpha_1,\beta_1,\alpha_2,\beta_2)
= ( -\alpha_1-\alpha_2, -\beta_1-\beta_2, \alpha_1\alpha_2 - \beta_1\beta_2, \alpha_1\beta_2+\beta_1\alpha_2),
$$
and its restriction to $\mathcal R_2(\bbd )$ (we labeled $\phi$) has the form
\begin{equation}\label{Eq-casen=2}
\phi (\alpha_1,\beta_1,\alpha_2,\beta_2)
= (-\alpha_1-\alpha_2, \alpha_1\alpha_2 - \beta_1\beta_2)
\end{equation}
since in $\mathcal R_2(\bbd )$, $(\alpha_1,\beta_1)=(\alpha_2,-\beta_2)$ or $\beta_1=\beta_2=0$. Therefore $\Psi\circ\pi (\mathcal R_2(\bbd )) = \mathcal C_2(\bbd )$  is the union of two subsets $\mathcal C_2(\bbd ) = C_{0,1}\cup C_{2,0}$ given by
$$C_{0,1}=\{(a_0,a_1),\
a_0 = -2\alpha\ ,\ 
a_1 = \alpha^2 + \beta^2,\
\alpha^2+\beta^2<1
,\ \hbox{for some}\ \alpha,\beta \in\bbr, \beta\neq 0\}\ ,\ \hbox{and}$$
$$C_{2,0}=\{(a_0,a_1),\ a_0 = -(\alpha_1+\alpha_2),\;
a_1 = \alpha_1\alpha_2,\
\alpha_1^2<1,\ \alpha_2^2<1
,\ \hbox{for some}\ \alpha_1,\alpha_2\in\bbr\}.
$$
The notation $C_{0,1}$ refers to the subset of coefficients corresponding to one complex root (and its conjugate), while the set $C_{2,0}$ corresponds to two real roots. We can make both $C_{0,1}$ and $C_{2,0}$ more explicit. In fact, they simplify to the following form:
$$C_{0,1}=\left\{(a_0,a_1),\
{a_0^2\over 4}< a_1 <1\right\}\ \ \hbox{and}\ \ C_{2,0}=\left\{(a_0,a_1),\ |a_0|-1\leq a_1\leq {a_0^2\over 4} \right\},
$$
and their union is precisely Eq. \eqref{c2B} of Example \ref{schurproblem}.
This can now be drawn in the coefficient plane $(a_0,a_1)$. The desired region is given in Fig. (\ref{Fig-TD}),
with $C_{0,1}$ being the dotted top part (\textit{without} the boundary), and $C_{2,0}$ being the solidly shaded bottom part, together with the boundary. Both subsets are disjoint, and their union is $\mathcal C_2(\bbd )$, which is the entire triangular region.
\end{example}
\section{The embedding theory}\label{Sec-Embedding}

In this section, we review the basic theory required to achieve the goals of this paper. The authors established this section's findings in \cite{Al-Ca-Ka2023}, and since these results are very recent, we extract the main results and refine them to fit our purpose. 

Let $F: \mathbb{R}^k\rightarrow \mathbb{R}$ denote a function that exhibits monotonicity in each argument in the conventional sense. Our concept of ``increasing" and ``decreasing" are meant to be non-decreasing and non-increasing, respectively.  Equip $\mathbb{R}^k$ with a partial ordering ``$\leq_\tau$" that is compatible with the monotonicity, meaning
$$X\leq_\tau Y\ \Longrightarrow\ F(X)\leq F(Y).$$
More precisely, we have the following definition: 

\begin{definition}\label{ordering}
Associate to the map of mixed monotonicity $F$ the 
map $\tau:\;\{1,2,\ldots,k\}\to \{-1,1\}$ which is $\tau (i)=1$ if $ F$ is increasing in its $ith$ argument, and $\tau (i)=-1$ if 
$F$ is decreasing in that argument. Define then the
partial ordering $\leq_\tau$ of $\bbr^k$  by
$$(x_1,x_2,\ldots,x_k)\leq_{\tau}(y_1,y_2,\ldots,y_k)\ \Longleftrightarrow\ x_i\leq y_i\ \hbox{if}\ \tau (i)=1,\ \hbox{and}\ x_i\geq y_i\ \hbox{if}\ \tau (i)=-1.$$
If a function
$F: \mathbb{R}^k\rightarrow \mathbb{R}$ is $\tau$-monotonic, we write $F(\uparrow_\tau)$.
\end{definition}

Let $\uparrow$ denote non-decreasing and $\downarrow$ denote non-increasing. Observe that if a function is of the form $F(\uparrow,\downarrow),$ then $\leq_\tau$ is the so-called ``south-east ($se$)" order, while if $F(\uparrow,\uparrow),$ then $\leq_\tau$ is the so-called ``north-east" order. $\tau$-Monotonic functions $F$ appear frequently in discrete dynamical systems of the form
\begin{equation}\label{Eq-General-DE}
x_{n+1}=F(x_n,x_{n-1},\ldots,x_{n-k+1}), \quad \text{where}\quad n\in \mathbb{N}=\mathbb{Z}^+\cup\{0\},
\end{equation}
 with initial conditions in $\mathbb{R}^k$ or $\mathbb{R}_+^k.$ A simple example, which will occupy us in most of this paper, is the \textit{linear} kth-order difference equation
\begin{equation}\label{Eq-General-LDE}
x_{n+1}= -a_{k-1}x_{n}-a_{k-2}x_{n-1}-\cdots-a_0x_{n-k+1}.
\end{equation}

When $\mathbb{R}^k$ is equipped with a $\tau$ order, any rectangular region in  $\mathbb{R}^k$ has a minimal and a maximal element. In particular, if $\Omega=[a,b]^k,$ then one of the $2^k$ vertices forms the minimal element, while another vertex forms the maximal element. To be more specific,  the minimal element is $m_\tau:=(u_1,u_2,\ldots,u_k),$ where $u_j=a$ if $F$ is increasing in the $jth$ argument, and $u_j=b$ otherwise. The maximal element $M_\tau$ is the ``dual" of $m_\tau$ obtained by switching the $a$'s and $b$'s.

Next, for given $x,y$, $x\leq y$, and $\tau$ as in Definition \ref{ordering}, we define  the point
\begin{equation}\label{pointer}
P_\tau(x,y) = (b_1,\ldots, b_k),\;\text{where}\; b_i = \begin{cases} x,&\ \hbox{if}\ \tau (i)=1\\
y,&\ \hbox{if}\ \tau (i)=-1.
\end{cases}\
\end{equation}
It will be convenient throughout to write $P_\tau=P_\tau(x,y)$, with $x,y$ understood.
The dual $P^t_\tau$ of $P_\tau$ is defined by replacing the $x$'s with $y$'s. If  the system
\begin{equation}\label{Eq-Pseudo}
(F(P_\tau), F(P^t_\tau)) =(x,y)
\end{equation}
 has a {\it unique} solution, it must be of the form $x=y=\bar x,$ which is a fixed point of $F.$  If $x\neq y$ is a solution of Eq. \eqref{Eq-Pseudo}, then $(x,y)$ and $(y,x)$ are dubbed \textit{pseudo-fixed points} of $F.$ We stress that the structure of $P_\tau$ changes based on whether $\tau(1)=1$ or $-1.$

\begin{example}\rm To illustrate the concept of pseudo-fixed points, consider
$$x_{n+1}=F_1(x_n,x_{n-1})=sx_n+(s-1)x_{n-1},\quad 0<s<1.$$
In this example, we have $F_1(\uparrow,\downarrow)$, i.e., $\tau(1)=1$ and $\tau(2)=-1.$ So $P_\tau = (x,y)$. Solving for  $$(x,y) = (F_1(P_\tau), F_1(P^t_\tau)) =(sx+(s-1)y,sy+(s-1)x)=(x,y),$$
we obtain the solution $\bar x=0$, which is the unique fixed point, together with the solution set of all points along the curve $y=-x$. Therefore, all points $(x,-x)$ with $x\neq 0$ are pseudo-fixed points.
\end{example}

The ensuing pivotal outcome is a consequence of the main results of \cite{Al-Ca-Ka2023}.

\begin{theorem}\label{Th-GlobalStability} \cite{Al-Ca-Ka2023} Consider Eq. \eqref{Eq-General-DE} in which $F$ is monotonic in each one of its arguments. Assume that $F$ has a unique fixed point and no pseudo-fixed points. If for each initial condition $X_0:=(x_0,x_{-1},\ldots,x_{-k+1})$ in the domain of $F,$ there exists a point $P_\tau=P_\tau(x,y)$ such that $x\leq y$, $P_\tau\leq_\tau X_0 \leq_\tau P_\tau^t,$  and the two inequalities 
\begin{equation}\label{theoremi}
x<F(P_\tau)\quad\text{and}\quad F(P^t_\tau)<y
\end{equation}
hold,
then the fixed point of $F$ is a global attractor.
\end{theorem}
We stress that the inequalities in \eqref{theoremi} are strict. 
The proof of Theorem \ref{Th-GlobalStability}, which can be found in (\cite{Al-Ca-Ka2023}, Corollary 2.12), consists in extending the System \eqref{Eq-General-DE} from $\bbr^k$ to a new system $G:\;\bbr^{k}\times\bbr^{k}\to \bbr^{k}\times\bbr^{k}$, with the property that $G$ is monotonic with respect to the southeast ordering ($({\bf x}_1, {\bf y}_1)\leq_{se} ({\bf x}_2, {\bf y}_2)$ if ${\bf x}_1\leq_\tau {\bf x}_2$ and ${\bf y}_2\leq_\tau {\bf y}_1$). The convergence of orbits of $G$ implies, under suitable conditions, the convergence of orbits of $F$. We refer the interested reader to \cite{Al-Ca-Ka2023} for more details. We call the set of points $(x,y)$ that satisfy $x\leq y$ and the inequalities $x<F(P_\tau),\;F(P^t_\tau)<y$ a {\it feasible} region. Observe that applying Theorem \eqref{theoremi} is a type of squeezing strategy in a boxed region.\\

Linear systems in Eq. \eqref{Eq-General-LDE}  fit well under the scope of Theorem \ref{Th-GlobalStability}. The origin is always the unique fixed.  In the next two sections, we apply Theorem \ref{Th-GlobalStability} to build a computational theory giving a sequence of sufficient conditions for Schur's problem, i.e., giving conditions on the coefficients to ensure that the roots are in $\mathbb{D}$. 

\section{Degree two polynomials}\label{motivational}

We start with the case of degree two polynomials to clarify our rationale and set the stage for the general method spelled out in \S\ref{algorithm}.

Start with the polynomial
\begin{equation}\label{Eq-PolynomialDegree2}
    p(x)=x^2-\alpha x+\beta.
\end{equation} 
The roots of $p$ are located within  $\mathbb{D}$ if and only if $(\alpha,\beta)$ are located within the triangular shaded region in Fig. \ref{Fig-TD}. This is an immediate consequence of the Jury's test for this polynomial or the computation of Example \ref{Ex-Sadok}. We will exhaust this region in stages using the dynamical system approach.

\definecolor{qqqqff}{rgb}{0.,0.,1.}
\definecolor{ffqqqq}{rgb}{1.,0.,0.}
\definecolor{cqcqcq}{rgb}{0.7529411764705882,0.7529411764705882,0.7529411764705882}
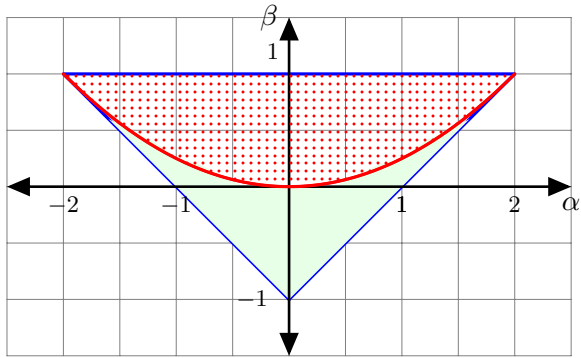
\begin{figure}[htpb]
\begin{center}
\begin{tikzpicture}[line cap=round,line join=round,>=triangle 45,x=1.0cm,y=1.0cm,scale=1.5]
\draw[help lines,step=.5] (-2.5,-1.5) grid (2.5,1.5);
\draw[line width=1.2pt,domain=0:10.0,smooth,variable=\x,color=blue] (0,-1)--(2,1)--(-2,1)--(0,-1);
\fill[fill=green!10] (-2,1)--(-1.8,0.8)--(-1.6,0.6)--(-1.4,0.5)--(-1.2,0.4)--(-1.0,0.25)--(-0.8,0.2)--(-0.6,0.1)--(-0.4,0.0)--(-0.2,0.0)--(0.0,0.0)--(0.2,0.0)--(0.4,0.0)--(0.6,0.1)--(0.8,0.2)--(1.0,0.25)--(1.2,0.4)--(1.4,0.5)--(1.6,0.6)--(1.8,0.8)--(2.0,1.0)  -- (0,-1) -- (-2,1)-- cycle;
\draw[-triangle 45, line width=1.0pt,scale=1] (0,0) -- (2.5,0) node[below] {$\alpha$};
\draw[line width=1.0pt,-triangle 45] (0,0) -- (-2.5,0);
\draw[-triangle 45, line width=1.0pt,scale=1] (0,0) -- (0,1.5) node[left] {$\beta$};
\draw[line width=1.0pt,-triangle 45] (0,0) -- (0.0,-1.5);
\draw [samples=50,domain=-2.0:2.0,xshift=0.0cm,yshift=0.0cm,line width=1.0pt,color=ffqqqq,fill=ffqqqq,pattern=dots,pattern color=ffqqqq] plot ({\x},{0.25*\x*\x});
\draw[line width=1.2pt,domain=-2:2,smooth,variable=\x,color=red] plot ({\x},{0.25*\x*\x)});
\draw[scale=1] (1,0) node[below] {\footnotesize $1 $};
\draw[scale=1] (0,1.2) node[left] {\footnotesize $1 $};
\draw[scale=1] (1,0.0) node[below] {\footnotesize $1 $};
\draw[scale=1] (-1,0) node[below] {\footnotesize $-1 $};
\draw[scale=1] (-0.1,-1) node[left] {\footnotesize $-1 $};
\draw[scale=1] (2,0) node[below] {\footnotesize $2 $};
\draw[scale=1] (-2,0) node[below] {\footnotesize $-2 $};
\end{tikzpicture}
\end{center}
\caption{The region with solid shading indicates where the eigenvalues are real and within $\mathbb{D}$. In contrast, the shaded region with dots indicates the region where the eigenvalues are non-real and within $\mathbb{D}$.}\label{Fig-TD}
\end{figure}

\subsection{The linear  approximation} 
As explained in the introduction, we associate to $p$ the discrete linear system
\begin{equation}\label{Eq-SecondOrder1}
x_{n+1}=F(x_n,x_{n-1})=\alpha x_n-\beta x_{n-1},\; x_0,x_{-1}\in \mathbb{R},\; n=0,1,\ldots
\end{equation}
The first observation to make is that the eigenvalues $\lambda_1$ and $\lambda_2$ of the Jacobian matrix for this system are precisely the roots of $p$ in Eq. \eqref{Eq-PolynomialDegree2}. In this case, we have
$$\alpha = \lambda_1 + \lambda_2\ \ ,\ \ 
\beta = \lambda_1\lambda_2.$$

We study the global stability of Eq. \eqref{Eq-SecondOrder1}, which is equivalent to having $\lambda_1,\lambda_2\in\mathbb D$.
Observe that $F$ is monotonic, and its monotonicity in each argument depends on the signs of $\alpha$ and $\beta$.  This allows us to consider the zero boundary cases of $\alpha$ and $\beta$ as positive or negative.  Therefore, there are four cases to consider:

\begin{description}
\item{\textbf{Case i:}} When $(\alpha,\beta)=(+,+),$ we have $F(\uparrow,\downarrow),$ $\tau(1)=1, \tau(2)=-1$ and $P_\tau = (x,y)$ as introduced in \eqref{pointer}.  By Theorem \ref{Th-GlobalStability}, to obtain global convergence, we must solve for $(x,y)$ in the inequalities \eqref{theoremi}
$$ x< F(x,y)=\alpha x-\beta y\quad \text{and}\quad y>F(y,x)=\alpha y-\beta x,$$
or equivalently,
$\ds y< \frac{(\alpha -1)}{\beta} x$ and $\ds x> \frac{(\alpha -1)}{\beta} y$, given $\alpha,\beta\in \bbr_+^2$. 
Recall the signs of $\alpha$ and $\beta$ together with the assumption that $x\leq y$. The \textit{solution set} of $(x,y)$ is non-empty if and only if
$$\alpha <1 \quad \text{and}\quad (\alpha-1)^2> \beta^2.$$
Since $\alpha,\beta\geq 0$, this is equivalent to the \textit{feasible region} given by 
\begin{equation}\label{In-Degree2-Cond1}
\alpha\geq 0,\;\beta\geq 0\quad \text{and}\quad \alpha+\beta <1.
\end{equation}
Under these conditions, the solution set for $P_\tau$ consists of all the points $(x,y)$ in the shaded region caught between the two lines sketched in Fig. \ref{feasible}.
We also have to explain why for all $X_0\in\bbr^2$, there exists a $P_\tau$ such that $P_\tau\leq_\tau X_0\leq_\tau P_\tau^t$ (this is the initial condition in Theorem \ref{Th-GlobalStability}, with $\tau$ the south-east ordering, i.e., $\tau(1)=1$ and $\tau(2)=-1$). Fig. \ref{feasible} shows why the initial condition holds: it is because the solution set is unbounded in the northwest directions as depicted. To clarify the choice of $P_\tau$ for an initial condition $X_0=(x_0,y_0),$ consider 
$$P_\tau=\left(\min\{x_0,y_0\}-\epsilon,\max\{x_0,y_0\}+\epsilon\right),$$
then $P_\tau\leq_\tau X_0 \leq_\tau P_\tau^t.$ In other words, $P_\tau$, and $P_\tau^t$, are the upper left and lower right (respectively) corners of a square containing $X_0.$ Fig. \ref{feasible} gives a visual illustration. Note that the lines representing the boundary of the feasible region in the figure are symmetric with respect to the $y=-x$ axis.
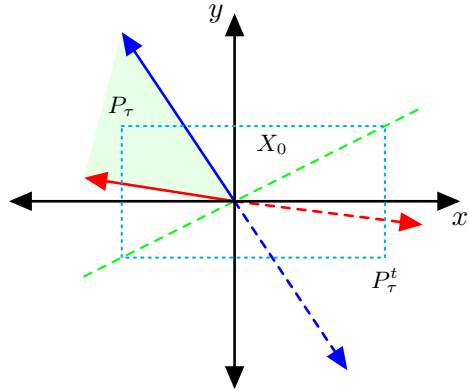
\begin{figure}[htb]
   \begin{center}
\begin{tikzpicture}[line cap=round,line join=round,>=triangle 45,x=1.0cm,y=1.0cm,yscale=0.5,scale=0.5]
\draw[line width=0.8pt,color=green,dashed] (-4,-4) -- (5.0,5.0);
\fill[fill=green!10] (0,0)--(-3.0,9)--(-4,1.25)--(0,0)-- cycle;
\draw[->,line width=1.0pt,color=blue] (0,0)--(-3,9);
\draw[->,line width=1.0pt,dashed,color=blue] (0,0)--(3,-9);
\draw[->,line width=1.0pt,color=red] (0,0)--(-4,1.25);
\draw[->,line width=1.0pt,color=red,dashed] (0,0)--(5,-1.25);
\draw[-triangle 45, line width=1.0pt,scale=1] (0,0) -- (6,0) node[below] {$x$};
\draw[line width=1.0pt,-triangle 45] (0,0) -- (-6,0);
\draw[-triangle 45, line width=1.0pt,scale=1] (0,0) -- (0,10) node[left] {$y$};
\draw[line width=1.0pt,-triangle 45] (0,0) -- (0.0,-10);
\draw[line width=0.8pt,color=cyan,dotted] (-3,4)--(4,4)--(4,-3)--(-3,-3)--(-3,4);
\draw[scale=1] (-3,4) node[above] {\footnotesize $P_\tau $};
\draw[scale=1] (4,-3) node[below] {\footnotesize $P_\tau^t $};
\draw[scale=1] (1,2.0) node[above] {\footnotesize $X_0 $};
\end{tikzpicture}
\end{center}
    \caption{The region with solid boundary is the solution set of Inequality \eqref{theoremi} for given $\alpha,\beta$ in the coefficient locus. For any initial condition $X_0,$ $P_\tau$ can be found so that $P_\tau\leq_\tau X_0\leq_\tau P_\tau^t$ (for the south-east ordering $\tau(1)=1,\tau(2)=-1$), because the ray on the left (resp. its reflection across the $x=y$ axis), is unbounded in the north-west (resp. south-east) directions. Observe that $P_\tau$ and $P_\tau^t$ are two opposite corners of a box containing $X_0$. Indeed, $P_\tau$  and $P_\tau$ form the minimal and maximal elements of the boxed region with respect to the $\leq_\tau$ partial order.}\label{feasible}
\end{figure}

In summary, for this special case, global convergence under condition \eqref{In-Degree2-Cond1} occurs (by Theorem \ref{Th-GlobalStability}), and it is to the unique fixed point $(0,0)$ of $F$ which is the only equilibrium solution of the linear system. Note that when it comes to pseudo-fixed points, or solutions to the equation $(x,y)= (F(x,y),F(y,x))$ with $x\neq y$, we obtain infinitely many solutions when $|\alpha-1|=|\beta|.$

\item{\textbf{Case ii:}} When $(\alpha,\beta)=(-,+),$ we have $F(\downarrow,\downarrow)$
and $P_\tau = (y,y)$. Here we resolve \eqref{theoremi} to obtain
$$ y> F(x,x)=\alpha x-\beta x\quad \text{and}\quad x< F(y,y)=\alpha y-\beta y,$$
or equivalently,
$y> (\alpha -\beta) x\quad \text{and}\quad x< (\alpha -\beta) y.$
Since $\alpha<0$ and $\beta>0,$ we have $\alpha-\beta$ is negative. We obtain a feasible solution from the assumption $x\leq y$ and the two obtained inequalities when $(\alpha-\beta)^2<1.$ Putting all the conditions on $\alpha$ and $\beta$ together, we arrive at
\begin{equation}\label{In-Degree2-Cond2}
\alpha\leq 0,\;\beta\geq 0\quad \text{and}\quad \beta-\alpha <1.
\end{equation}
We obtain
infinitely many pseudo-fixed points when $\alpha +1=\beta$.

\item{\textbf{Case iii:}} When $(\alpha,\beta)=(+,-),$ we have $F(\uparrow,\uparrow)$ and $P_\tau = (x,x)$. As in the previous two cases, we find infinitely many pseudo-fixed points when $\alpha-1=\beta.$ On the other hand, the conditions on $\alpha$ and $\beta$ that guarantee a feasible region for the system of inequalities are
$$x\leq y,\quad x< (\alpha-\beta)x\quad \text{and}\quad (\alpha-\beta)y< y.$$
This leads us to the following conditions on $\alpha$ and $\beta:$
\begin{equation}\label{In-Degree2-Cond3}
\alpha\geq 0,\;\beta\leq 0\quad \text{and}\quad \alpha-\beta <1.
\end{equation}
\item{\textbf{Case iv:}} Finally,  when $(\alpha,\beta)=(-,-),$ we have $F(\downarrow,\uparrow)$ and $P_\tau = (y,x)$. Solving \eqref{Eq-Pseudo} gives infinitely many pseudo-fixed points when $|\alpha|=|\beta+1|.$ Similar to the previous cases, the inequalities in \eqref{theoremi} yield the following conditions on $\alpha$ and $\beta$ (the feasible region)
\begin{equation}\label{In-Degree2-Cond4}
\alpha\leq 0,\;\beta\leq 0\quad \text{and}\quad \beta+\alpha >-1.
\end{equation}
\end{description}
 The results from the four cases mentioned above lead us to the $\ell_1$-condition.

\begin{corollary}\label{casen=2}
A stability region for the discrete system $x_{n+1}=\alpha x_n-\beta x_{n-1}$ in the parameter space is given by the inequality $|\alpha| + |\beta|<1$.
\end{corollary}

\begin{proof}
All previous cases (i)-(iv) are special cases of this inequality.
\end{proof}

The region obtained by Corollary \ref{casen=2} is depicted in Fig. \ref{Fig-TD2}, and Table \ref{Tab-TD1} breaks it down more precisely based on the signs of coefficients.  This region is naturally strictly included in
the region of necessary and sufficient conditions of Fig. \ref{Fig-TD}.

\begin{remark}\rm 
The pseudo-fixed points we encounter render our squeezing strategy invalid under Theorem \ref{Th-GlobalStability}, so we do not include the boundary of the region obtained by corollary \ref{casen=2}.
\end{remark}

\begin{table}[ht]
 \footnotesize
\caption{It is important to note that in cases (i) and (ii), we stated that the conditions are both necessary and sufficient. However, we determined the necessity based on our prior knowledge, which will be confirmed in section five. }
 \centering
 \begin{tabular}{l c c c c c c c l} 
 \hline\hline
 &&&&&&&&\\   [-1.0ex]
 Case & &Signs & & Monotonicity & &Conditions& &Remarks \\ [1.0ex]
 \hline \hline
 &&&&&&&&\\   [-1.0ex]
 (i) && $\alpha>0,\;\beta>0$ && $F(\uparrow,\downarrow)$ && $\beta+\alpha<1$ && Sufficient \\ [1ex]
 (ii) & &$\alpha<0,\;\beta>0$ && $F(\downarrow,\downarrow)$  && $\beta-\alpha<1$&& Sufficient  \\ [1ex]
 (iii) && $\alpha>0,\;\beta<0$ && $F(\uparrow,\uparrow)$   && $\alpha-\beta<1$ && Necessary and sufficient \\ [1ex]
  (iv) && $\alpha<0,\;\beta<0$ && $F(\downarrow,\uparrow)$  && $\beta+\alpha>-1$ && Necessary and sufficient \\ [1ex]
  \hline\hline
  \end{tabular} \label{Tab-TD1}
  \end{table}

\definecolor{qqqqff}{rgb}{0.,0.,1.}
\definecolor{ffqqqq}{rgb}{1.,0.,0.}
\definecolor{cqcqcq}{rgb}{0.7529411764705882,0.7529411764705882,0.7529411764705882}
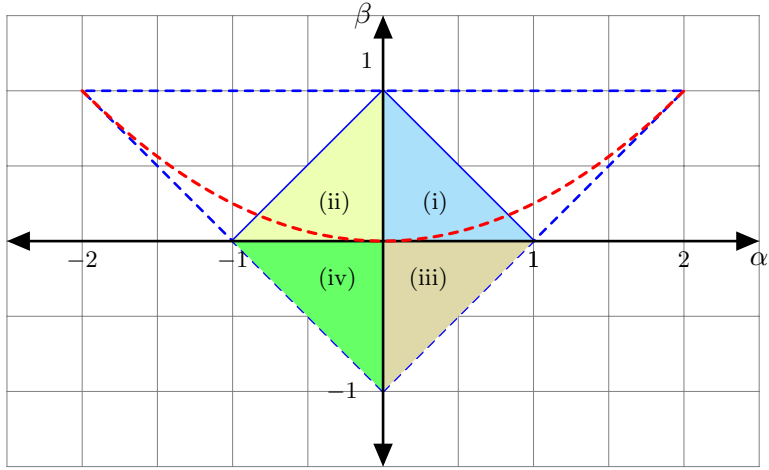
\begin{figure}[htpb]
\begin{center}
\begin{tikzpicture}[line cap=round,line join=round,>=triangle 45,x=1.0cm,y=1.0cm,scale=2.0]
\draw[help lines,step=.5] (-2.5,-1.5) grid (2.5,1.5);
\draw[line width=1.2pt,color=blue] (1,0)--(0,1)--(-1,0);
\draw[line width=1.0pt,color=blue,dashed] (0,-1)--(1,0)--(2,1)--(-2,1)--(-1,0)--(0,-1);
\fill[fill=cyan!30] (0,0)--(1,0)--(0,1)--(0,0);
\fill[fill=lime!30] (0,0)--(0,1)--(-1,0)--(0,0);
\fill[fill=green!60] (0,0)--(-1,0)--(0,-1)--(0,0);
\fill[fill=olive!30] (0,0)--(1,0)--(0,-1)--(0,0);
\draw[-triangle 45, line width=1.0pt,scale=1] (0,0) -- (2.5,0) node[below] {$\alpha$};
\draw[line width=1.0pt,-triangle 45] (0,0) -- (-2.5,0);
\draw[-triangle 45, line width=1.0pt,scale=1] (0,0) -- (0,1.5) node[left] {$\beta$};
\draw[line width=1.0pt,-triangle 45] (0,0) -- (0.0,-1.5);
\draw[line width=1.2pt,domain=-2:2,smooth,variable=\x,color=red,dashed] plot ({\x},{0.25*\x*\x)});
\draw[scale=1] (1,0) node[below] {\footnotesize $1 $};
\draw[scale=1] (0,1.2) node[left] {\footnotesize $1 $};
\draw[scale=1] (1,0.0) node[below] {\footnotesize $1 $};
\draw[scale=1] (-1,0) node[below] {\footnotesize $-1 $};
\draw[scale=1] (-0.1,-1) node[left] {\footnotesize $-1 $};
\draw[scale=1] (2,0) node[below] {\footnotesize $2 $};
\draw[scale=1] (-2,0) node[below] {\footnotesize $-2 $};
\draw[scale=1] (0.5,0.25) node[left] {\footnotesize (i)};
\draw[scale=1] (-0.5,0.25) node[right] {\footnotesize (ii)};
\draw[scale=1] (-0.5,-0.25) node[right] {\footnotesize (iv)};
\draw[scale=1] (0.5,-0.25) node[left] {\footnotesize (iii)};
\end{tikzpicture}
\end{center}
    \caption{This figure illustrates the conditions obtained on $\alpha$ and $\beta$ in contrast with Fig. \ref{Fig-TD}. The cases (i) to (iv) are based on Table \ref{Tab-TD1}. The colors reflect different types of monotonicity in Definition \ref{ordering}}\label{Fig-TD2}
\end{figure}

From the above cases and conditions \ref{theoremi} in Theorem \ref{Th-GlobalStability}, it becomes evident that our approach narrows down to tackling two linear inequalities in the region $x<y.$ In particular, we face the following system of linear inequalities:
\begin{equation}\label{(I)}
(I)\qquad \begin{cases}
x\leq& sx+ty\\
y\geq& tx+sy,
\end{cases}
\end{equation}
under different scenarios for the parameters $s$ and $t.$
For convenience in the sequel, we summarize the cases in which we obtain a feasible region in the following lemma. The proof is simple and omitted.

\begin{lemma}\label{Lem-FeasibleRegion}
Assume $x\leq y$ and consider the system of inequalities in (I). Each of the following holds true:
\begin{description}
\item{(i)} Consider  $s=0.$  There exists a feasible region iff $t>-1$
\item{(ii)} Consider  $t=0.$   There exists a feasible region iff $s<1$
\item{(iii)} Consider  $st< 0.$   A feasible region exists iff $|s|+|t|<1.$
\end{description}
\end{lemma}
It is worth mentioning that the three cases listed in Lemma \ref{Lem-FeasibleRegion} can be condensed in the inequality $|s|+|t|<1$; however, the separation between the cases is done to simplify the application based on the signs of the coefficients (or the monotonicities of $F$ in Eq. \eqref{Eq-SecondOrder1}). Note that Case (i) is associated with $F(\downarrow,\downarrow),$ Case (ii) is associated with $F(\uparrow,\uparrow)$, and Case (iii) is associated with $F(\uparrow,\downarrow)$ or $F(\downarrow,\uparrow)$ when $s$ and $t$ are switched. Also, $s=0$ means the coefficients in $F$ are non-positive, and $t$ combines $\alpha$ and $\beta$. Similarly, $s$ combines $\alpha$ and $\beta$ when $t=0$.

\subsection{Increasing delay}
As evidenced by the observations made in cases (i) and (ii) of Table 1, the obtained criterion is only sufficient, and the stability region in Fig. (\ref{Fig-TD}) is not completely covered. To enlarge the feasible domain for stability, i.e., add more sufficient conditions, we can appeal to a higher-degree equation as follows: from Eq. \eqref{Eq-SecondOrder1}, substitute $x_n=\alpha x_{n-1}-\beta x_{n-2}$ to obtain
\begin{equation}\label{Eq-SecondOrder2}
y_{n+1}=(\alpha^2-\beta) y_{n-1}-\alpha\beta y_{n-2}=F_2(y_n,y_{n-1},y_{n-2}).
\end{equation}
We changed $x_n$ into $y_n$ to stress that we are dealing with a new equation here. We increased the delay in Eq. \eqref{Eq-SecondOrder2} by one, and every solution of Eq. \eqref{Eq-SecondOrder1} can be a solution of Eq. \eqref{Eq-SecondOrder2} under suitable choices of the initial conditions. In particular, when $x_{-1}$ and $x_0$ are the initial conditions of Eq. \eqref{Eq-SecondOrder1}, we can consider $y_{-2}=(\alpha x_{-1}-x_0)/\beta$, $y_{-1}=x_{-1}$ and $y_0=x_0$ as the initial conditions in Eq. \eqref{Eq-SecondOrder2} to show that a solution of Eq. \eqref{Eq-SecondOrder1} can be a solution of \eqref{Eq-SecondOrder2}. Note that when we talk about a fixed point of $F$ or one of its expansions $F_j,$ we write the scalar value $\bar x;$ however, when we talk about an equilibrium solution of the $F_j$-difference equation, $\bar x$ denotes a vector with $j+1$ components.  In particular, $\bar x=0$ is a fixed point for both $F$ in Eq. \eqref{Eq-SecondOrder1} and $F_2$ in Eq. \eqref{Eq-SecondOrder2}, while $\bar x=(0,0)$ and $\bar x=(0,0,0)$ are equilibrium solutions of Eqs. \eqref{Eq-SecondOrder1} and \eqref{Eq-SecondOrder2}, respectively. Up to this end, it becomes clear that global convergence in Eq. \eqref{Eq-SecondOrder2} leads to a global convergence in Eq. \eqref{Eq-SecondOrder1}. Therefore, stability conditions on the new coefficients $\alpha^2-\beta$ and $\alpha\beta$ add alternative sufficient conditions to ensure that the zeros of $p$ in Eq. \eqref{Eq-PolynomialDegree2} are within the unit disk.
Observe that $F_2$ is constant in its first argument, while the monotonicity in the second and third arguments depends on the signs of $\alpha^2-\beta$ and $\alpha\beta,$ respectively. Therefore, we proceed to handle each case:
\begin{description}
\item{ \textbf{Case (i) revisited:}} Consider $\alpha>0,\beta>0$ and $\beta>\alpha^2.$ In this case, $F_2$ is constant in its first component. So, we can consider  $F_2(\uparrow,\downarrow,\downarrow)$ or $F_2(\downarrow,\downarrow,\downarrow).$ Both options lead to the same conclusion. So, we consider the latter option, and in this case, $P_\tau = (y,y,y)$. The system \eqref{theoremi} in $x,y$ takes the form
$x < F(y,y,y) = (\alpha^2-\beta-\alpha\beta)y$ and
$y> F(x,x,x) = (\alpha^2 - \beta -\alpha\beta)x$. By Lemma \ref{Lem-FeasibleRegion} (i), a solution exists if
$$\alpha^2-\beta-\alpha\beta+1>0.$$
On the other hand, if we consider $\alpha>0,\beta>0$ and $\beta<\alpha^2,$ then we have $P_\tau = (x,x,y), x<y$ and to have a feasible solution for System \eqref{theoremi}, we need the extra assumption $\alpha <1.$ Hence, from the two scenarios, we obtain
\begin{equation}\label{Eq-InCase(i)Revisted}
0\leq \alpha<1\quad \text{and}\quad 0<\beta<\frac{1+\alpha^2}{1+\alpha}.
\end{equation}

\item{ \textbf{Case (ii) revisited:}} Consider $\alpha<0,\beta>0$ and $\beta>\alpha^2.$ In this case, $F_2(\downarrow,\downarrow,\uparrow)$, $P_\tau=(y,y,x)$ and to obtain a feasible solution for the system \eqref{theoremi}, we need
$$\alpha^2+1>\beta(1-\alpha).$$
On the other hand, if we consider $\alpha<0,\beta>0$ and $\beta<\alpha^2,$ then $F_2(\uparrow,\uparrow,\uparrow)$, $P_\tau = (x,x,x)$ and we need the extra assumption $\alpha>-1$  coming from
\eqref{theoremi}. From the two scenarios again, we obtain
\begin{equation}\label{Eq-InCase(ii)Revisted}
-1<\alpha\leq 0\quad \text{and}\quad \beta<\frac{1+\alpha^2}{1-\alpha}.
\end{equation}
Based on these results, we update Fig. \ref{Fig-TD2} to obtain Fig. \ref{Fig-TD3}. Similarly, Table \ref{Tab-TD1} can be updated. Observe that when $\beta>0,$ the sufficient conditions $\alpha+\beta<1$ and $\beta-\alpha<1$ are updated to become
\begin{equation}\label{In-RevisedSuffCond1}
-1<\alpha<1\quad \text{and}\quad \beta<\frac{1+\alpha^2}{1+|\alpha|},
\end{equation}
which expands the stability region as shown in Fig. \ref{Fig-TD3}.
\definecolor{qqqqff}{rgb}{0.,0.,1.}
\definecolor{ffqqqq}{rgb}{1.,0.,0.}
\definecolor{cqcqcq}{rgb}{0.7529411764705882,0.7529411764705882,0.7529411764705882}
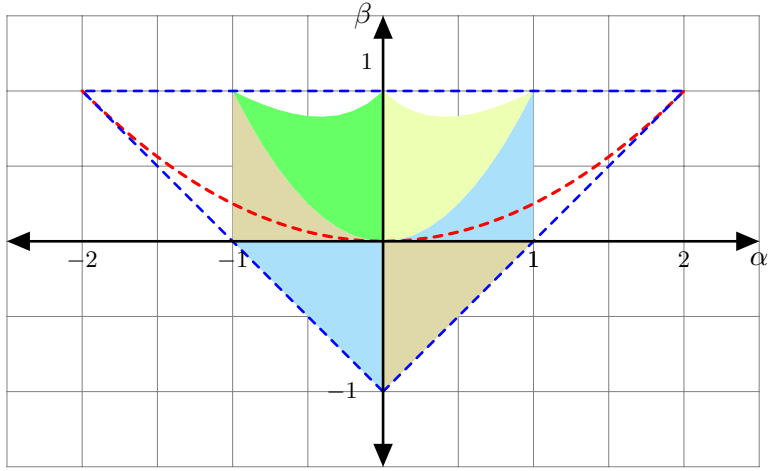
\begin{figure}[htpb]
\begin{center}
\begin{tikzpicture}[line cap=round,line join=round,>=triangle 45,x=1.0cm,y=1.0cm,scale=2.0]
\draw[help lines,step=.5] (-2.5,-1.5) grid (2.5,1.5);
\fill[fill=cyan!30,smooth] (0,0)--(0.1,0.01)--(0.2,0.04)--(0.3,0.09)--(0.4,0.16)--(0.5,0.25)--(0.6,0.36)--(0.7,0.49)--(0.8,0.64)--(0.9,0.81)--(1,1)--(1,0)--(0,0)--cycle;
\fill[fill=olive!30,smooth] (0,0)--(-0.1,0.01)--(-0.2,0.04)--(-0.3,0.09)--(-0.4,0.16)--(-0.5,0.25)--(-0.6,0.36)--(-0.7,0.49)--(-0.8,0.64)--(-0.9,0.81)--(-1,1)--(-1,0)--(0,0)--cycle;
\fill[fill=lime!30,smooth] (0,0)--(0.1,0.01)--(0.2,0.04)--(0.3,0.09)--(0.4,0.16)--(0.5,0.25)--(0.6,0.36)--(0.7,0.49)--(0.8,0.64)--(0.9,0.81)--(1,1)--(0.9,0.953)--(0.8,0.911)
--(0.7,0.876)--(0.6,0.85)--(0.5,0.833)--(0.4,0.829)--(0.3,0.838)--(0.2,0.867)--(0.1,0.918)--(0,1)--(0,0)--cycle;
\fill[fill=green!60,smooth] (0,0)--(-0.1,0.01)--(-0.2,0.04)--(-0.3,0.09)--(-0.4,0.16)--(-0.5,0.25)--(-0.6,0.36)--(-0.7,0.49)--(-0.8,0.64)--(-0.9,0.81)--(-1,1)--(-0.9,0.953)--(-0.8,0.911)
--(-0.7,0.876)--(-0.6,0.85)--(-0.5,0.833)--(-0.4,0.829)--(-0.3,0.838)--(-0.2,0.867)--(-0.1,0.918)--(0,1)--(0,0)--cycle;
\draw[line width=1.2pt,domain=-2:2,smooth,variable=\x,color=red,dashed] plot ({\x},{0.25*\x*\x)});
\fill[fill=cyan!30] (0,0)--(-1,0)--(0,-1)--(0,0);
\fill[fill=olive!30] (0,0)--(1,0)--(0,-1)--(0,0);
\draw[line width=1.0pt,color=blue,dashed] (0,-1)--(1,0)--(2,1)--(-2,1)--(-1,0)--(0,-1);
\draw[-triangle 45, line width=1.0pt,scale=1] (0,0) -- (2.5,0) node[below] {$\alpha$};
\draw[line width=1.0pt,-triangle 45] (0,0) -- (-2.5,0);
\draw[-triangle 45, line width=1.0pt,scale=1] (0,0) -- (0,1.5) node[left] {$\beta$};
\draw[line width=1.0pt,-triangle 45] (0,0) -- (0.0,-1.5);
\draw[scale=1] (1,0) node[below] {\footnotesize $1 $};
\draw[scale=1] (0,1.2) node[left] {\footnotesize $1 $};
\draw[scale=1] (1,0.0) node[below] {\footnotesize $1 $};
\draw[scale=1] (-1,0) node[below] {\footnotesize $-1 $};
\draw[scale=1] (-0.1,-1) node[left] {\footnotesize $-1 $};
\draw[scale=1] (2,0) node[below] {\footnotesize $2 $};
\draw[scale=1] (-2,0) node[below] {\footnotesize $-2 $};
\end{tikzpicture}
\end{center}
\caption{This figure illustrates the feasible region obtained by the conditions on $\alpha$ and $\beta$ after revisiting Case (i) and Case (ii) of table \ref{Tab-TD1}. The colors of the shaded regions match the colors in Fig. \ref{Fig-TD2} to reflect the type of monotonicity we obtain after one step of the expansion. The total stability region expands due to the inequalities in \eqref{In-RevisedSuffCond1}, as compared to Fig \ref{Fig-TD2}.}\label{Fig-TD3}
\end{figure}
\end{description}
The above process can be continued to expand the region of the coefficients $\alpha$ and $\beta$ in the upper two quadrants. Write $s_0=\alpha$ and $t_0=-\beta,$ then $s_1=s_0^2+t_0$ and $t_1=t_0s_0.$ In general, this leads to tackling
\begin{equation}\label{iteration}
x_{n+1}=F_{j+1}(x_n,x_{n-1},\ldots, x_{n-j-1})=s_jx_{n-j}+t_jx_{n-j-1},
\end{equation}
where $s_{j}=s_0s_{j-1}+t_{j-1}$ and $t_{j}=t_0s_{j-1}.$  Based on the sequences $\{s_j\},$ $\{t_j\}$ and Lemma \ref{Lem-FeasibleRegion}, we reach the following conclusion:

\begin{proposition}\label{Th-Order2}
Let $s_0=\alpha$ and $t_0=-\beta.$ Define the recursive sequences $s_{j}=s_0s_{j-1}+t_{j-1}$ and $t_{j}=t_0s_{j-1}.$ If $|s_j|+|t_j|<1$ for some $j=0,1,\ldots$, then the zero equilibrium in Eq. \eqref{Eq-SecondOrder1} is globally attracting.
\end{proposition}

Observe that Proposition \ref{Th-Order2} translates the sufficient conditions we obtained in our approach into a recurrence process. Let $S_j$ represent the feasible region of the inequality $|s_j|+|t_j|<1.$ For instance, when $j=2,$ we obtain
\begin{equation}\label{Eq-SecondOrder3}
x_{n+1}=F_3(x_n,x_{n-1},x_{n-2},x_{n-3})=s_2x_{n-2}+t_2x_{n-3},
\end{equation}
where $s_2=\alpha(\alpha^2-2\beta)$ and $ t_2=-\beta(\alpha^2-\beta).$
Applying Proposition \ref{Th-Order2} gives us the region $S_2$ described by the inequality
\begin{equation}\label{In-ThirdExpansion}
|\alpha(\alpha^2-2\beta)|+|\beta(\alpha^2-\beta)|<1.
\end{equation}
It can be seen graphically and computationally that $S_0\subset S_1,S_0\subset S_2,$ but $S_1\not\subset S_2.$ However, $S_1\cup S_2$ gives an overall improvement of $S_0,S_1$ and $S_2.$ The computational argument is as follows: $S_2$ is the region bounded between the curves of
\begin{align*}
\alpha=&1\\
\beta=& \alpha^2-|\alpha|+1\\
\beta^2=&\alpha(\alpha-2)\beta+(\alpha-1)(\alpha^2+\alpha+1)\\
\beta^2=&\alpha(\alpha+2)\beta+(\alpha+1)(\alpha^2-\alpha+1).
\end{align*}
The curve $\beta= \alpha^2-|\alpha|+1$ does not improve the curve obtained from the inequality boundary in \eqref{In-RevisedSuffCond1}; therefore, we can ignore it. The improved region becomes $S_1\cup S_2,$ which is illustrated in Fig. \ref{Fig-TD4}.
\definecolor{qqqqff}{rgb}{0.,0.,1.}
\definecolor{ffqqqq}{rgb}{1.,0.,0.}
\definecolor{cqcqcq}{rgb}{0.7529411764705882,0.7529411764705882,0.7529411764705882}
\begin{figure}[htpb]
\begin{center}
\begin{tikzpicture}[line cap=round,line join=round,>=triangle 45,x=1.0cm,y=1.0cm,scale=2.0]
\draw[help lines,step=.5] (-2.5,-1.5) grid (2.5,1.5);
\fill[fill=green!60,smooth]  (0,0)--(0.1,0.01)--(0.2,0.04)--(0.3,0.09)--(0.4,0.16)--(0.5,0.25)--(0.6,0.36)--(0.7,0.49)--(0.8,0.64)--(0.9,0.81)--(1,1)--
(0.9,0.95)--(0.8,0.91)--(0.7,0.88)--(0.6,0.85)--(0.5,0.83)--(0.4,0.83)--(0.3,0.84)--(0.2,0.87)--(0.1,0.92)--(0,1)--(0,0)--cycle;
\fill[fill=olive!30,smooth]  (0,0)--(-0.1,0.01)--(-0.2,0.04)--(-0.3,0.09)--(-0.4,0.16)--(-0.5,0.25)--(-0.6,0.36)--(-0.7,0.49)--(-0.8,0.64)--(-0.9,0.81)--(-1,1)--
(-0.9,0.95)--(-0.8,0.91)--(-0.7,0.88)--(-0.6,0.85)--(-0.5,0.83)--(-0.4,0.83)--(-0.3,0.84)--(-0.2,0.87)--(-0.1,0.92)--(0,1)--(0,0)--cycle;
\fill[fill=lime!30,smooth] (0,0)--(0.1,0.01)--(0.2,0.04)--(0.3,0.09)--(0.4,0.16)--(0.5,0.25)--(0.6,0.36)--(0.7,0.49)--(0.8,0.64)--(0.9,0.81)--(1,1)--
(1.414,1)--(1.4,0.98)--(1.3,0.845)--(1.2,0.72)--(1.1,0.605)--(1,0.5)--(0.9,0.405)--(0.8,0.32)--(0.7,0.245)--(0.6,.18)--(0.5,0.125)--(0.4,0.08)--(0.3,0.045)--(0.2,0.02)--(0.1,0.005)--(0,0)--cycle;
\fill[fill=cyan!30,smooth] (1.414,1)--(1.4,0.98)--(1.3,0.845)--(1.2,0.72)--(1.1,0.605)--(1,0.5)--(0.9,0.405)--(0.8,0.32)--(0.7,0.245)--(0.6,.18)--(0.5,0.125)--(0.4,0.08)--(0.3,0.045)--(0.2,0.02)--(0.1,0.005)--(0,0)--(1,0)--(1.414,1)--cycle;
\fill[fill=cyan!30,smooth] (0,0)--(-0.1,0.01)--(-0.2,0.04)--(-0.3,0.09)--(-0.4,0.16)--(-0.5,0.25)--(-0.6,0.36)--(-0.7,0.49)--(-0.8,0.64)--(-0.9,0.81)--(-1,1)--
(-1.414,1)--(-1.4,0.98)--(-1.3,0.845)--(-1.2,0.72)--(-1.1,0.605)--(-1,0.5)--(-0.9,0.405)--(-0.8,0.32)--(-0.7,0.245)--(-0.6,.18)--(-0.5,0.125)--(-0.4,0.08)--(-0.3,0.045)--(-0.2,0.02)--(-0.1,0.005)--(0,0)--cycle;
\fill[fill=lime!30,smooth] (-1.414,1)--(-1.4,0.98)--(-1.3,0.845)--(-1.2,0.72)--(-1.1,0.605)--(-1,0.5)--(-0.9,0.405)--(-0.8,0.32)--(-0.7,0.245)--(-0.6,.18)--(-0.5,0.125)--(-0.4,0.08)--(-0.3,0.045)--(-0.2,0.02)--(-0.1,0.005)--(0,0)--(-1,0)--(-1.414,1)--cycle;
\draw[line width=1.2pt,domain=-2:2,smooth,variable=\x,color=red,dashed] plot ({\x},{0.25*\x*\x)});
\fill[fill=green!60] (0,0)--(-1,0)--(0,-1)--(0,0);
\fill[fill=olive!30] (0,0)--(1,0)--(0,-1)--(0,0);
\draw[line width=1.0pt,color=blue,dashed] (0,-1)--(1,0)--(2,1)--(-2,1)--(-1,0)--(0,-1);
\draw[-triangle 45, line width=1.0pt,scale=1] (0,0) -- (2.5,0) node[below] {$\alpha$};
\draw[line width=1.0pt,-triangle 45] (0,0) -- (-2.5,0);
\draw[-triangle 45, line width=1.0pt,scale=1] (0,0) -- (0,1.5) node[left] {$\beta$};
\draw[line width=1.0pt,-triangle 45] (0,0) -- (0.0,-1.5);
\draw[scale=1] (1,0) node[below] {\footnotesize $1 $};
\draw[scale=1] (0,1.2) node[left] {\footnotesize $1 $};
\draw[scale=1] (1,0.0) node[below] {\footnotesize $1 $};
\draw[scale=1] (-1,0) node[below] {\footnotesize $-1 $};
\draw[scale=1] (-0.1,-1) node[left] {\footnotesize $-1 $};
\draw[scale=1] (2,0) node[below] {\footnotesize $2 $};
\draw[scale=1] (-2,0) node[below] {\footnotesize $-2 $};
\end{tikzpicture}
\end{center}
\caption{ The total shaded region in this figure illustrates the feasible region of Inequality \ref{In-ThirdExpansion}. The colored regions illustrate the type of monotonicity we obtain in the second expansion. Again here, the coloring matches the coloring in Fig. \ref{Fig-TD2} and Fig. \ref{Fig-TD3} to reflect the type of monotonicity in the expanded map. }\label{Fig-TD4}
\end{figure}
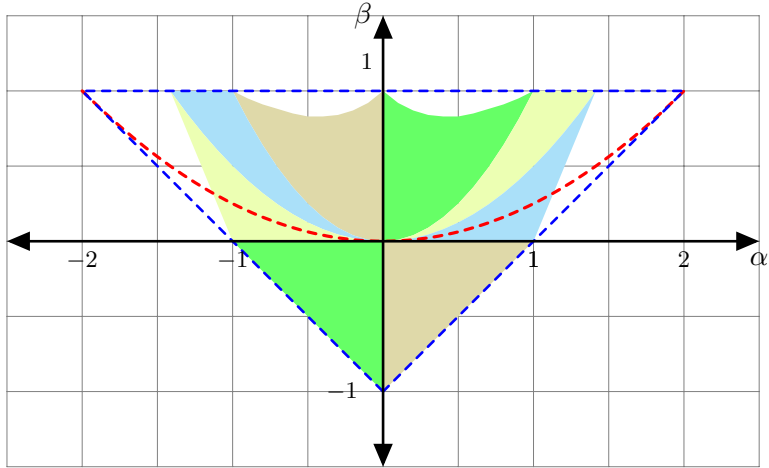

The recurrence process in Proposition \ref{Th-Order2} can be computationally implemented, as shown in Fig.  \ref{Fig-ForFun}. The figure shows the stability region in the $(\alpha,\beta)$-plane for up to $(s_3,t_3)$.

\begin{figure}[h!]
\centering
 \includegraphics[width=0.5\linewidth]{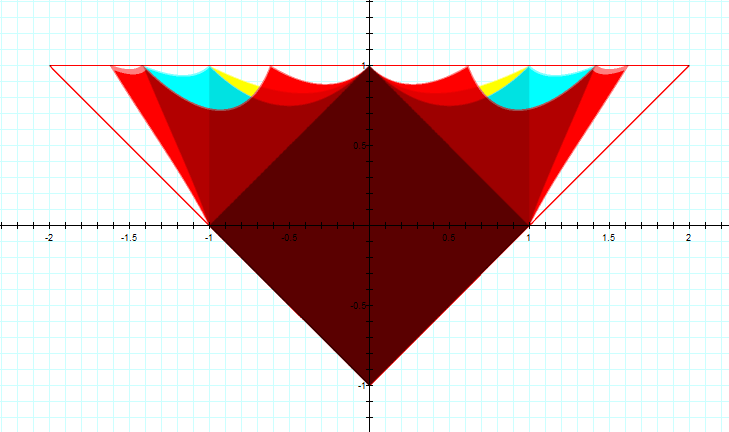}
\caption{This figure shows the stability regions in the $(\alpha,\beta)$-plane obtained after computationally applying  Proposition \ref{Th-Order2} for $(s_0,t_0)$ to $(s_3,t_3)$. Here, the colors differ from those in Fig. \ref{Fig-TD2} to Fig. \ref{Fig-TD4} because this figure's goal is to illustrate how the stability region extends overall, but not necessarily from one step to the next.}
\label{Fig-ForFun}
\end{figure}
The above inspiring example is consistently emphasized due to its ability to reveal significant critical ideas in developing our general algorithm. Observe that no potential for improvement was possible when the value of $\beta$ was negative in Eq. \eqref{Eq-SecondOrder1}. However, when $\beta$ was positive, we enhanced the stability region by employing the same technique on the iterations. What is the reason for this phenomenon? When $(\alpha,\beta)=(+,-)$ in Case (iii), the map $F$ in $x_{n+1}=F(x_n,x_{n-1})$ had infinitely many pseudo-fixed points when $\beta-\alpha=-1.$ The same condition continues to give infinitely many pseudo-fixed points for the iterated map, i.e., the map of $x_{n+1}=F(F(x_{n-1},x_{n-2}),x_{n-1})$. The same scenario happened in Case (iv) when $(\alpha,\beta)=(-,-).$ This phenomenon crippled the process, and this lines up with the fact that the obtained conditions are necessary and sufficient when $\beta\leq 0$.

\section{The Algorithm}\label{algorithm}

After illustrating how to implement our approach in the simple case of degree-two polynomials, we proceed to tackle the general case. 
Start with the $n^{th}$-degree monic-polynomial
as in Eq. \eqref{polynomial}
  \begin{equation}\label{Eq-nthDegreePolynomial}
  p(x)=x^n+a_{n-1}x^{n-1}+a_{n-2}x^{n-2}+\cdots+a_1x+a_0,
  \end{equation}
and associate to it the $n^{th}$-order linear difference equation
 \begin{equation}\label{Eq-nthDegreeDE2}
  x_{k+1}=F(x_k,x_{k-1},\ldots,x_{k-n+1})=-a_{n-1}x_k-a_{n-2}x_{k-1}-\cdots-a_0x_{k-n+1}.
  \end{equation}
The monotonicity of $F$ depends on the signs of the coefficients, and in general, there are $2^n$ cases to consider.
The following key proposition is mostly classical and can be traced back to Smithies \cite{smithies}. It provides Schur stability under inexpensive constraints when sufficient conditions are sought. The proof we provide is novel compared to the established complex analysis method that employs Rouch\'e's theorem. 

\begin{proposition}\label{Pr-MainTheorem} ($\ell_1$-condition).
 Consider the polynomial $p(x)$ in Eq. \eqref{Eq-nthDegreePolynomial} with real coefficients $a_j$.
Then the zeros of $p(x)$ are located within $\mathbb{D}$, that is $p(x)$ is Schur-stable, if its $\ell_1$-norm  is less than two, i.e.
$$\|p\|_1=1+\sum_{j=0}^{n-1}|a_j|< 2\ \Longrightarrow\ p(x)\ \hbox{is Shur stable}$$
This condition is necessary and sufficient if:\\
(i) All the coefficients are negative.\\
(ii) For each $0\leq k\leq n-1$, $(-1)^{k+n}a_k < 0$.\\
In particular, the $\ell_1$-condition is necessary and sufficient in the case $n=3$ if $a_1< 0$ and $a_2a_0> 0$.
\end{proposition}

\begin{proof}
   Define the sign function as follows:
  $$\sign(a_j)= \begin{cases}
  1& \text{if}\quad  a_j>0\\
  0& \text{if}\quad  a_j=0\\
  -1& \text{if} \quad a_j<0.
  \end{cases}
  $$
  and let $\Gamma_j$ denote the set of indices $i$ in which $\sign(a_i)=(-1)^{j+1} $ for $j=1,2.$
  The $\leq_\tau$ partial order of Definition \ref{ordering} is defined based on $\sign(a_j).$ If $\sign(a_j)\geq 0,$  the map $F$ in Eq. \eqref{Eq-nthDegreeDE2} is decreasing in its $j^{th}$ component and $\tau(j)=-1$, otherwise, it is increasing and $\tau(j)=1$. Now,  let $x\leq y,$ and consider $P_\tau$ as in \eqref{pointer}.  
  By Theorem \ref{Th-GlobalStability}, we have to check the inequalities 
$x\leq F(P_\tau)$, $F(P_\tau^t)\leq y$ and the initial condition $P_\tau\leq_\tau X_0\leq P_\tau^t$, for any $X_0\in\bbr^{n}$. The first condition follows from the fact that upon replacing $P_\tau$ into that inequality, we end up with System (I) in \eqref{(I)}.  
  If $\sign(a_{n-1})=1,$ the inequalities become seeking a feasible solution for a system of the form
  $$x\leq sx+ty\quad \text{and}\quad y\geq sy+tx,$$
  where
  $\ds s=\sum_{i\in\Gamma_1}a_i$ and $\ds t=\sum_{i\in \Gamma_2}a_i. $
  On the other hand, If $\sign(a_{n-1})=-1,$ we need a feasible solution for
  $$tx+sy\leq y\quad \text{and}\quad x\leq ty+sx.$$
  Without loss of generality, consider the first case when $\sign(a_{n-1})=1$.
  Based on Lemma \ref{Lem-FeasibleRegion}, a solution exists if 
  \begin{equation} \label{step1}|\sum_{i\in\Gamma_1}a_i| + |\sum_{i\in \Gamma_2}a_i| < 1.
  \end{equation}
  However, since $\Gamma_j$ classifies the coefficients based on their signs, we obtain  
  $\displaystyle \left|\sum_{i\in\Gamma_j}a_i\right|= \sum_{i\in\Gamma_j} |a_i|.$ Therefore, it follows that if 
  $\displaystyle\sum_{i\in\Gamma_1} |a_i| + \sum_{i\in\Gamma_2} |a_i| = \sum_{i=0}^{n-1}|a_i| <1$, and \eqref{step1} is verified. On the other hand, the initial condition is satisfied for the same reasons as in Case (i) of Section \ref{motivational}, that is, for the unboundedness of the set of solutions of $P_\tau$ in rays codified by the $\tau$-ordering (see Fig. \ref{feasible}).
  We can now invoke Theorem \ref{Th-GlobalStability} to obtain global convergence to the zero equilibrium. This conclusion implies that the zeros of the polynomial $p(x)$ are located inside $\mathbb{D}$.

  The necessity is a consequence of (Theorem II page 269 of \cite{smithies}). Smithies shows that if the polynomial is $\ds p(x) = x^n - \sum_{k=0}^{n-1}\alpha_kx^k$, then necessary conditions that all the roots are less than
unity in absolute value are  (i) $\ds\sum_{k=0}^{n-1}\alpha_k<1$, and (ii)
  $\ds (-1)^n\sum_{k=0}^{n-1} (-1)^k\alpha_k<1$.
  Each condition is necessary. When each of this sum is equal to $\ds\sum_{k=0}^{n-1} |\alpha_i|$, this also gives sufficiency. This happens under the conditions stated in the Proposition. Conditions (i) and (ii) can be condensed into one condition when $n=3.$ 
  \end{proof}

 As observed in our motivational example and its consequent Proposition \ref{Th-Order2}, the outcome of Proposition \ref{Pr-MainTheorem} can be enhanced by passing to iterate constructions, i.e., the stability region may be extended by applying the $\ell_1$-condition on the iterates. To formalize this notion, we let $a_{n-j}$ be the first nonzero coefficient in Eq. \eqref{Eq-nthDegreeDE2}. This means $p$ becomes 
\begin{equation}\label{Eq-(an-j)}
    p(x)=x^n+a_{n-j}x^{n-j}+a_{n-j-1}x^{n-j-1}+\cdots+a_1x+a_0\ \ ,\ \ a_{n-j}\neq 0
\end{equation}
   and its associated linear system becomes  
   \begin{equation}\label{Eq-(Xn-j)}x_{k+1} = F(x_k,x_{k-1},\ldots, x_{k-n+1})=
   -a_{n-j}x_{k-j+1}-a_{n-j-1}x_{k-j}-\cdots-a_0x_{k-n+1}.
   \end{equation}
\begin{theorem}\label{Th-Algorithm} (Algorithm).
   Consider the real coefficient polynomial $p$ in Eq. \eqref{Eq-(an-j)} and its associated linear system in Eq.  \eqref{Eq-(Xn-j)}. 
    Define the iterated system by substituting $F$ for the first variable of nonzero coefficient (i.e., $x_{k-j+1}$)  
   \begin{eqnarray}
   y_{k+1} &=& 
   -a_{n-j}F(y_{k-j},\ldots, y_{k-n-j+1})-a_{n-j-1}y_{k-j}-\cdots-a_0y_{k-n+1}\nonumber\\
   &=:&b_{j}y_{k-j} + \cdots +
   b_{j+n-1}y_{k-j-n+1}.\label{Eq-higherorder2}
   \end{eqnarray}
   Then
   $$\sum_{i=j}^{j+n-1} |b_i| < 1\ \Longrightarrow\  p(x)\ \hbox{is Schur stable}.$$
   (The same result holds if we substitute in any other variable).
\end{theorem}

The explicit form of each $b_i$ is given in terms of the coefficients of $p$ after the substitution. The algorithm
now starts with the polynomial $p$ under scrutiny and computes $\|p\|_1.$ If $\|p\|_1-1 = \sum |a_i| < 1,$ the roots are within the
open unit disk, and $p$ is Schur stable. If not, we consider the iterated system in Eq. \eqref{Eq-higherorder2} obtained by substitution. 
If  $\sum |b_i| < 1,$ then $p$ is Schur stable, and the process stops. If not, then we iterate again by substituting
$y_{k-j} = F(y_{k-j-1}, ..., y_{k-j-n-1})$ in Eq. \eqref{Eq-higherorder2} and check if the new sum $\sum |c_i|$ of the positive part of the 
coefficients of this new linear system is less than one or not. If the sum at any iterate is less than one, then we obtain a sufficient condition for the Schur stability of $p.$  Section \ref{Sec-Applications} illustrates the process on a number of examples.

\begin{proof}
As already mentioned before, every solution of Eq. \eqref{Eq-(Xn-j)} can be a solution of Eq. \eqref{Eq-higherorder2} under suitable choices of the initial conditions. Therefore, convergence to zero in Eq. \eqref{Eq-higherorder2} forces convergence to zero in Eq. \eqref{Eq-nthDegreeDE2}. Now,
the condition $\displaystyle\sum_{i=j}^{j+n-1} |b_i| < 1$ ensures that the iterated system converges to zero by Proposition \ref{Pr-MainTheorem}, which guarantees the Schur stability of $p.$ 
\end{proof}

Even though our method for addressing Schur stability relies on sufficient conditions, it provides several advantages over the Jury's stability algorithm. In the Jury's algorithm, the stability decision cannot be determined until the conclusion of the procedure. In contrast, our approach has the advantage of making decisions regarding stability from the outset in some cases. 

\begin{remark}\rm There are three necessary conditions in the Jury's stability test, namely $|p(0)|<1,$ $p(1)>0$, and $(-1)^np(-1)>0.$ When those conditions guarantee the $\ell_1$-condition, we obtain the necessary and sufficient conditions in Proposition \ref{Pr-MainTheorem}. This was clear in the third and fourth quadrants of Fig. \ref{Fig-TD2} (see also Table \ref{Tab-TD1}. For $n^{th}$ degree polynomials in general, the $\ell_1$-condition of the first step is necessary and sufficient for two orthants of the $n$-dimentional parameter space. On the other hand, those conditions are implicitly involved in the hypothesis of Proposition \ref{Pr-MainTheorem}. The condition $|p(0)|=|a_0|<1$ is obvious. So, we clarify the other two conditions: $$p(1)=1+\sum_{j=0}^{n-1}a_j\geq 1-\sum_{j=0}^{n-1}|a_j|\geq 1-1=0$$
 and
 $$(-1)^np(-1)=1-a_{n-1}+a_{n-2}+\ldots+(-1)^na_0>1-\sum_{j=0}^{n-1}|a_j|\geq 1-1=0.$$
 Therefore, the Jury's necessary conditions are guaranteed by $\|p\|_1<2.$ The point here is that checking these conditions before implementing Proposition \ref{Pr-MainTheorem} and Theorem \ref{Th-Algorithm} can be labor-saving when parameters are involved in the coefficients.
 \end{remark}

\section{Running the algorithm: applications}\label{Sec-Applications}

In this section, we provide three illustrative examples demonstrating the utility and effectiveness of Theorem \ref{Th-Algorithm}. We begin with a simple numerical example compared with the Jury's algorithm. 

 \begin{example}\label{Ex-1}\rm
 \textbf{(A numerical comparison with the Jury's algorithm):}
Consider the polynomial 
\begin{equation}\label{Eq-Example-p(x)}
p(x)=x^5+\frac{1}{2}x^4-\frac{1}{2}x-\frac{1}{2}.
\end{equation}
\end{example}
This polynomial has a real root between zero and one; however, we want to check whether all roots are clustered in the open unit disk.  
The necessary conditions are satisfied since $|p(0)|=\frac{1}{2}<1$ and $p(1)=(-1)^5p(-1)=\frac{1}{2}>0.$ Now, we run the Jury's algorithm and summarize the computations in Table \ref{Tab-Example6.1A}. Then, we run our algorithm and summarize the computations in Table \ref{Tab-Example6.1B}.

\begin{table}[ht]
\caption{This table shows the computations of the Jury's algorithm when applied on the polynomial $p(x)$ in Eq. \eqref{Eq-Example-p(x)}. }\label{Tab-Example6.1A}
 \centering
\begin{tabular}{c| c c c c c c}
\hline\hline
 &&& &&& \\   [-1.0ex]
\textbf{ Step} & $x^0$ & $x^1$ & $x^2$ & $x^3$ & $x^4$ & $x^5$ \\[1.0ex]
  \hline\hline
  &&&&&&\\  [-1.0ex]
  \textbf{1} & $-\frac{1}{2}$ & $-\frac{1}{2}$ & $0$ & $0$ & $\frac{1}{2}$ & $1$ \\  [1.0ex]
  \textbf{2} & $1$ & $\frac{1}{2}$ & $0$ & $0$ & $-\frac{1}{2}$ & $-\frac{1}{2}$ \\  [1.0ex]
  \hline
  &&&&&&\\  [-1.0ex]
  \textbf{3}&$-\frac{3}{4}$ & $-\frac{3}{4}$ & $0$ & $0$ & $\frac{3}{4}$ &   \\ [1.0ex]
  \textbf{4} & $\frac{3}{4}$ & $0$ & $0$ & $-\frac{3}{4}$ & $-\frac{3}{4}$ & \\  [1.0ex]
  \hline
  &&&&&&\\  [-1.0ex]
  \textbf{5} & $\frac{35}{16}$ & $\frac{3}{16}$ & $0$ & $\frac{1}{16}$ &  &  \\ [1.0ex]
  \textbf{ 6} & $\frac{1}{16}$ & $0$ & $\frac{3}{16}$  & $\frac{35}{16}$ &  &  \\  [1.0ex]
  \hline
  &&&&&&\\  [-1.0ex]
  \textbf{ 7} & $\frac{153}{32}$ & $\frac{105}{256}$ & $\frac{-3}{256}$  &  &  &  \\  [1.0ex]
  \hline\hline
\end{tabular}
\end{table}

\begin{table}[ht]
\caption{This table shows the computations of our algorithm when applied on the linear system associated with the polynomial $p(x)$ in Eq. \eqref{Eq-Example-p(x)}. The polynomials in the table are the polynomials associated with linear systems obtained from algorithm of Theorem \ref{Th-Algorithm}.}\label{Tab-Example6.1B}
 \centering
\begin{tabular}{c| c c c c}
\hline\hline
 &&&&\\ [-1.0ex]
\textbf{ Step} &&$p(x)$ && $\|p\|_1$ \\ [1.0ex]
  \hline\hline
  &&&&\\
  \textbf{1} & &$x^5+\frac{1}{2}x^4-\frac{1}{2}x-\frac{1}{2}$ &&  $1+\frac{3}{2}$ \\
 &&&&\\
  \textbf{2}& &$x^6-\frac{1}{4}x^4-\frac{1}{2}x^2-\frac{1}{4}x+\frac{1}{4}$ & & $1+\frac{5}{4}$ \\
 &&&&\\
   \textbf{3}& &$x^7+\frac{1}{8}x^4-\frac{1}{2}x^3-\frac{1}{4}x^2+\frac{1}{8}x-\frac{1}{8}$ &&  $1+\frac{9}{8}$ \\
 &&&&\\
  \textbf{4}& &$x^8-\frac{9}{16}x^4-\frac{1}{4}x^3+\frac{1}{8}x^2-\frac{1}{16}x+\frac{1}{16}$ &&  $1+\frac{17}{16}$ \\
  &&&&\\
  \textbf{5}& &$x^9+\frac{1}{32}x^4+\frac{1}{8}x^3-\frac{1}{16}x^2-\frac{7}{32}x-\frac{9}{32}$ &&  ${\bf 1+\frac{23}{32}<2}$ \\
  &&&&\\
  \hline\hline
\end{tabular}
\end{table}

\begin{example}\label{Ex-2}\rm
\textbf{A Cournot oligopoly model:}  Theocharis proposed the following model \cite{Th1960}, and examined the impact of increasing the number of competitors on stability in a Cournot oligopoly characterized by a linear demand function and constant marginal costs.
\begin{equation}\label{Eq-Theocharis}
x_{n+1}^i=\max\left\{0,(1-\lambda)x^i_n+\lambda\left(\frac{a-c_i}{2b}-\frac{1}{2}\sum_{j=1}^N(x_n^j-x_n^i)\right)\right\},
\end{equation}
where $x_n^i$ is the output of the $ith$ competitor at discrete time unit $n,$ $N$ is the total number of competitors, $c_i$ represents the constant marginal cost of competitor $i,$ $a$ and $b$ are parameters related to the price, and $\lambda \in[0,1]$ is the adaptive parameter.  Under the assumption that there is a delay of knowledge each firm has on its competitors, C\'{a}novas \cite{Ca2023}    considered Eq. \eqref{Eq-Theocharis} with delays
\begin{equation}\label{Eq-Theocharis}
x_{n+1}^i=\max\left\{0,(1-\lambda)x^i_n+\lambda\left(\frac{a-c_i}{2b}-\frac{1}{2}\sum_{j=1}^N(x_{n-k}^j-x_{n-k}^i)\right)\right\}.
\end{equation}
When the number of competitors $N=2,$ it was found in \cite{Ca2023} that some of the eigenvalues of the Jacobian matrix are the zeros of the polynomial
$$
p_1(x)=x^{k+1}-(1-\lambda)x^k-\frac{\lambda}{2},
$$
while the other eigenvalues are the zeros of the polynomial
$$
p_2(x)=x^{k+1}-(1-\lambda)x^k+\frac{\lambda}{2}.
$$
The local stability can be settled in two steps based on Proposition \ref{Pr-MainTheorem}. Indeed, for the polynomial $p_1,$ we have 
$$\|p_1\|_1=1+1-\lambda+\frac{\lambda}{2}<2,$$
and for the polynomial $p_2,$ we have
$$\|p_2\|_1=2-\lambda+\frac{\lambda}{2}<2.$$
Note that Jury's stability algorithm is tedious here, and its computations are formidable when $k$ is large.

When the number of competitors $N=3,$ the polynomial $p_1$ stays the same, while $p_2$ changes into
$$
p_3(x)=x^{k+1}-(1-\lambda)x^k+\lambda.
$$
We already know that the eigenvalues of $p_1$ are within  $\mathbb{D}$, and we need to test the polynomial $p_3.$ Since $0<\lambda<1,$ we have one coefficient positive and the other is negative. However, we obtain
$$\|p_3\|_1=1+1-\lambda+\lambda=2.$$
Therefore, Proposition \ref{Pr-MainTheorem} fails here, and this is because the map $F$ in
$$x_{n+1}=F(x_n,x_{n-1},\ldots,x_{n-k})=(1-\lambda)x_n-\lambda x_{n-k}$$
is having infinitely many pseudo-fixed points. So, we appeal to Theorem \ref{Th-Algorithm}. The first iterate gives us
\begin{align*}
x_{n+1}=&F_2(x_{n-1},\ldots,x_{n-k-1})\\
=&(1-\lambda)^2x_{n-1}-\lambda x_{n-k}-\lambda (1-\lambda)x_{n-k-1}
\end{align*}
and
$$\sum|a_j|=(1-\lambda)^2+\lambda+\lambda(1-\lambda)=1.$$
We keep repeating the process to obtain 
\begin{align*}
x_{n+1}=&F_{k+1}(x_{n-k},\ldots,x_{n-2k})\\
=&\left((1-\lambda)^{k+1}-\lambda\right)x_{n-k}-\lambda\sum_{j=1}^k (1-\lambda)^jx_{n-k-j},
\end{align*}
and consequently,
\begin{align*}
\sum|a_j|=&\left|(1-\lambda)^{k+1}-\lambda\right|+\lambda\sum_{j=1}^k(1-\lambda)^j\\
<&(1-\lambda)^{k+1}+\lambda+\lambda\sum_{j=1}^k(1-\lambda)^j=1.
\end{align*}
Hence, all roots of $p_3$ are clustered within the open unit disk. 
\end{example}

We close this section by another illustrative example from mathematical biology.

\begin{example}\rm
Consider the $N$-species competition model of Ricker type 
\begin{equation}\label{Eq-CompetitionModel}
x^j_{n+1}=x^j_n\exp\left[r_j-\sum_{i=1}^N a_{ji}x^i_n\right],\quad j=1,2,\ldots,N,
\end{equation}
where $r_j,\;j=1,\ldots,N$ are positive and represent the carrying capacities of species $x_n^j,j=1,\ldots,N,$  respectively. The other parameters are all nonnegative real numbers. When $N=2,$ the stability analysis of this model was done in  \cite{El-Lu2011,Lu-El-Ol2011}. We assume in this example that all carrying capacities are equal, i.e., $r_j=r$ for all $j.$ Also, a substitution can neutralize $a_{jj}$ to become $1$ for all $j.$ In other words, $a_{jj}=1.$ A coexistence equilibrium $(\bar x^1,\bar x^2,\ldots,\bar x^N)$ must be a solution of the linear system $AX=B,$ where 
$$
A=\left[
    \begin{array}{ccccc}
      1 & a_{12} & a_{13}&\ldots&a_{1N} \\
      a_{21} &1& a_{23} &\ldots&a_{2N}  \\
      \vdots & \vdots & \vdots&\ddots&\vdots \\
       a_{N1} & a_{N2} &\ldots&a_{N,N-1}&1  \\
    \end{array}
  \right],\quad  B=r\left[
                                    \begin{array}{c}
                                      1 \\
                                      1 \\
                                      \vdots\\
                                      1 \\
                                    \end{array}
                                  \right] \quad \text{and}\quad
                                  X=\left[
                                    \begin{array}{c}
                                      \bar x^1 \\
                                      \bar x^2 \\
                                      \vdots\\
                                      \bar x^N \\
                                    \end{array}
                                  \right].
$$
To guarantee the uniqueness of the coexistence, we consider $A$ to be invertible, i.e., $\det(A)\neq 0.$ It is logical to foster species so they can have the same coexistence level. In this case, we can force  $\bar x^j=t$ for all $j=1,\ldots,N.$ This means $\frac{r}{t}$ must be an eigenvalue of $A$ and $(1,\ldots,1)$ must be its associated eigenvector. In this case, we must have 
$$\sum_{i=1}^Na_{ji}=\sum_{i=1}^Na_{ki}=:a\quad \text{and}\quad t=\frac{r}{a}.$$

Next, the Jacobian matrix $J$ at the coexistence equilibrium $(t,\ldots,t)$ is given by 
$$J=\left[
    \begin{array}{ccccc}
      1-t & ta_{12} & ta_{13}&\ldots&ta_{1N} \\
      ta_{21} &1-t& ta_{23} &\ldots&ta_{2N}  \\
      \vdots & \vdots & \vdots&\ddots&\vdots \\
       ta_{N1} & ta_{N2} &\ldots&ta_{N,N-1}&1-t  \\
    \end{array}
  \right]. $$
  
  \begin{proposition}\label{Pr-Competition}
  Suppose the coexistence equilibrium satisfies $\bar x^j=t$ for all $j=1,\ldots,N.$ Then $\lambda_1=1-r$ is an eigenvalue of the Jacobian matrix $J.$ 
  \end{proposition}
  \begin{proof}
  Observe that $J=I-tA.$ Since $\bar x^j=t$ for all $j=1,\ldots,N,$  we already know that $\frac{r}{t}$ is an eigenvalue of $A$ and $V=(1,\ldots,1)$ is the eigenvector. Now, we obtain 
  $$JV=(I-tA)V=V-tA=V-t\frac{r}{t}V=(1-r)V.$$
  Therefore, $\lambda_1=1-r$ is an eigenvalue and $V=(1,\ldots,1)$ is an eigenvector belongs to $\lambda_1.$  
  \end{proof}
Since we aim to give showcases, we fix $N=3.$ The general case can be tackled similarly. Based on Lemma  \ref{Pr-Competition}, we write the characteristic polynomial of $J$ as 
$$p(x)=(x-r+1)(x^2-(r+2-3t)x+1+(a-3)t+b t^2, $$
where $a=c_1+c_2+1$ and $b=(a_3-c_1-1)b_1+c_1a_2-b_3+1.$ Recall that $t=\frac{r}{a},$ and therefore, the characteristic polynomial coefficients are written in terms of $r,a$ and $b.$ Based on these facts and our algorithm, we give the following result:

\begin{proposition}\label{Pr-MathBiologyExample}
Consider Model \ref{Eq-CompetitionModel} with $N=3$, and let $b=(a_3-c_1-1)b_1+c_1a_2-b_3+1.$ Assumptions that $r_1=r_2=r_3=r$ and $c_1+c_2=b_1+b_3=a_2+a_3=:a-1.$ The coexistence equilibrium $(t,t,t),$ where $t=\frac{r}{a}$ is locally asymptotically stable if $0<r<2$ and one of the following sufficient conditions is satisfied:\\
(i)  $|r + 2 - 3t| + |1 + (a - 3)t + bt^2|<1$\\
(ii)  $|(r + 2 - 3t)^2 - 1 - (a - 3)t - bt^2 + |r + 2 - 3t|\,|1 + (a - 3)t + bt^2|<1.$
\end{proposition} 
\begin{proof}
As verified in Proposition \ref{Pr-Competition}, $0<r<2$ makes one of the eigenvalues between $-1$ and $1.$ Now, condition (i) is from the $\ell_1$-norm of $x^2-(r+2-3t)x+1+(a-3)t+b t^2.$ Condition (ii) is from the $\ell_1$-norm of the first iterate. See Fig. \ref{Fig-MathBiology}  for illustrations in the $(b,a)$ plane as $r$ varies between zero and two. 
\end{proof}
We emphasize that we have set $N=3$ to avoid formidable expressions; nonetheless, the same algorithm can be applied in higher dimensions when no necessary and sufficient criteria are established. 
\end{example}
\definecolor{MyMaroon}{rgb}{128,0,0}
\definecolor{MyOrange}{rgb}{255,87,51}
\definecolor{ffvvqq}{rgb}{1,0.3333333333333333,0}
\definecolor{qqqqff}{rgb}{0.,0.,1.}
\definecolor{ffqqqq}{rgb}{1.,0.,0.}
\definecolor{cqcqcq}{rgb}{0.7529,0.7529,0.7529}
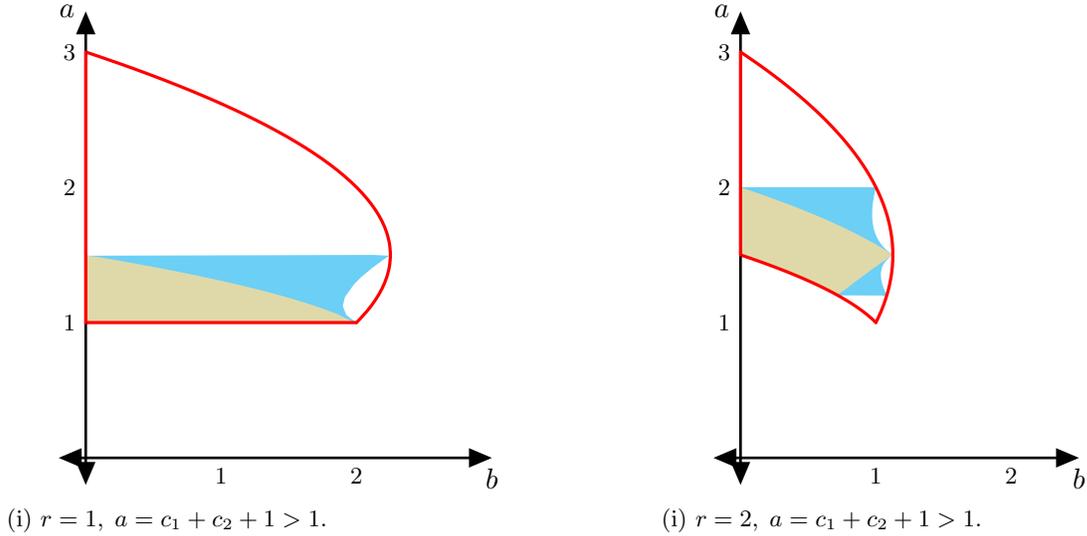
\begin{figure}[h!]
\centering
\begin{minipage}[t]{0.5\textwidth}
\begin{center}
\begin{tikzpicture}[line cap=round,line join=round,>=triangle 45,x=1.0cm,y=1.0cm,scale=1.8]
\fill[line width=1pt,color=ffqqqq,fill=cyan!50,fill opacity=1.0] (-5.392656115992917E-10,1.0050742321394215) -- (1.989769348797636,1.0053086419753108) -- (1.924671278408555,1.0553769543927087) -- (1.8984733217738352,1.12218806893875) -- (1.9137834864574463,1.1874999996223883) -- (1.9812655571603577,1.2830992736716789) -- (2.0348121113525472,1.336139034949472) -- (2.1181363877783115,1.40574221531788) -- (2.190307795965052,1.4590107197263422) -- (2.2413300962007123,1.494197530864199) -- (2.1221398634744966,1.50) -- (0,1.494197530864199) -- cycle;
\fill[line width=1pt,color=ffqqqq,fill=olive!30,fill opacity=1.0] (-1.7992123591753777E-8,1.0000000062829968) -- (1.9896430727023358,0.9974074074074083) -- (1.857168450635136,1.063381372961457) -- (1.7579765292005458,1.100721927475422) -- (1.6706072431270258,1.1305892656265788) -- (1.5992761547332894,1.1533372829039026) -- (1.4947317355454852,1.1845308573730806) -- (1.4027614041230687,1.210227822221491) -- (1.3395569174252362,1.2270857079735755) -- (1.2814363975858034,1.2420781459255426) -- (1.2238313983270581,1.2564999023780428) -- (1.1313673977138894,1.2788271494354626) -- (1.0000166236853427,1.3090132734471795) -- (0.8945700551903832,1.3321146681347722) -- (0.7990379826046115,1.352279424110142) -- (0.7049700725580993,1.3714961602503353) -- (0.6416615050720996,1.3841014279220512) -- (0.5985444378352116,1.3925448513988268) -- (0.5341697832032937,1.4049485154813985) -- (0.45404771795962917,1.4200657212775132) -- (0.39804580922489874,1.43043261603005) -- (0.3259729330225808,1.4435465134041838) -- (0.25268840181767377,1.45663137114814) -- (0.19768849837293206,1.4662945476818114) -- (0.148,1.47495) -- (0.0974,1.484) -- (0.0445,1.49) -- (0,1.5) -- cycle;
\draw[line width=1.2pt,domain=1:3.0,smooth,variable=\y,color=red] plot ({\y*(3-\y)},{\y});
\draw[line width=1.2pt,domain=0:2.0,smooth,variable=\x,color=red] plot ({\x},1);
\draw[-triangle 45, line width=1.0pt,scale=1] (0,0) -- (3.0,0) node[below] {$b$};
\draw[line width=1.0pt,-triangle 45] (0,0) -- (-0.2,0);
\draw[-triangle 45, line width=1.0pt,scale=1] (0,0) -- (0,3.3) node[left] {$a$};
\draw[line width=1.2pt,domain=1:3.0,smooth,variable=\y,color=red] plot (0,{\y});
\draw[line width=1.0pt,-triangle 45] (0,0) -- (0.0,-0.2);
\draw[scale=1] (1,0) node[below] {\footnotesize $1 $};
\draw[scale=1] (2,0) node[below] {\footnotesize $2 $};
\draw[scale=1] (0,1.0) node[left] {\footnotesize $1 $};
\draw[scale=1] (0,2.0) node[left] {\footnotesize $2 $};
\draw[scale=1] (0,3.0) node[left] {\footnotesize $3 $};
\draw[scale=1] (0.6,-0.3) node[below] {\footnotesize (i) $r=1,\; a=c_1+c_2+1>1.$};
\end{tikzpicture}
\end{center}
\end{minipage}%
\begin{minipage}[t]{0.5\textwidth}
\begin{center}
\begin{tikzpicture}[line cap=round,line join=round,>=triangle 45,x=1.0cm,y=1.0cm,scale=1.8]
\fill[line width=1pt,color=ffqqqq,fill=cyan!50,fill opacity=1.0] (1,2) -- (0.991780286475253,1.964887159896192) -- (0.985325867383787,1.9328563769879956) -- (0.9788214365828924,1.892753902090679) -- (0.9752672626610832,1.8629686660883893) -- (0.973039837003531,1.8348518468838557) -- (0.9720031392711901,1.7981354741948774) -- (0.9727411723059103,1.771909541050093) -- (0.9748828596445653,1.7457494572611385) -- (0.982096064798283,1.7022108598233525) -- (0.9888377225924476,1.6766695233440991) -- (1.0021003528446562,1.6408739815064917) -- (1.0150670758654923,1.614894688236794) -- (1.0334877621305818,1.586055695221658) -- (1.0508436839308046,1.5641966225444714) -- (1.0732650066054186,1.5408752436726694) -- (1.096803631000733,1.520504736534877) -- (1.1249950318432305,1.5031521918626725) -- (1.1024817472442419,1.4835084043793678) -- (1.0729187348995568,1.4549879109573767) -- (1.0541138338243785,1.4280183573010057) -- (1.043130247284786,1.4028796366074463) -- (1.0360497800470045,1.3644059736224117) -- (1.0382436221446696,1.3241855088687742) -- (1.046479068828602,1.287304725869354) -- (1.0584797910940322,1.2514490079672291) -- (1.0703142049402734,1.222077940342937) -- (1.077523183956556,1.1918545277196366) -- (1.0495473226030514,1.2000000005171354) -- (0.7178180185384202,1.2012105829477981) -- (0.6874556545680595,1.2177308769228574) -- (0.6428321089391331,1.2410029995643312) -- (0.6006410169107606,1.2620346583213613) -- (0.5602575912440695,1.281386115838861) -- (0.5275344599630665,1.29656451648804) -- (0.4968576351467629,1.310420541612402) -- (0.4536109978009819,1.329391493233066) -- (0.42392129450275073,1.3420636399696257) -- (0.3627801742201786,1.3673410023067762) -- (0.31829472796279046,1.3851004924666934) -- (0.2754027125647639,1.4017657885066483) -- (0.2281164141168438,1.4196579682171708) -- (0.1895558232763668,1.4339021058461576) -- (0.14722923496204332,1.4492033928693409) -- (0.10902759482176916,1.4627315070997051) -- (0.06440847034852941,1.4782140933883845) -- (0.0,1.502561391512716) -- (0.0,2.0) -- cycle;
\fill[line width=1pt,color=ffqqqq,fill=olive!30,fill opacity=1.0] (0.0,1.5) -- (0.05959704294886459,1.4798640137214383) -- (0.11787740207948971,1.4596205316648154) -- (0.17507232274814388,1.439176202887134) -- (0.23374852555883996,1.41755204832326) -- (0.2973103434038368,1.3933077244197225) -- (0.35316581610135755,1.3712222564825902) -- (0.418649079874024,1.3442856720996965) -- (0.4909423854331221,1.3130531123112978) -- (0.5468082681765615,1.2876763579624078) -- (0.6124955931711903,1.2562135946558577) -- (0.7135922582438369,1.2035458861880257) -- (0.7761088543717403,1.245880296641194) -- (0.8436022716905748,1.2989243789156522) -- (0.9933206826001817,1.4094826605541813) -- (1.0616876357033835,1.4571805883693028) -- (1.1249995493452725,1.5009493731906598) -- (1.0534396953336738,1.5456252100314962) -- (0.9902926186285707,1.5829278885480447) -- (0.9223909722418997,1.6205422511435041) -- (0.8250161487292605,1.6708123679132112) -- (0.7519177880963244,1.7062021505507234) -- (0.6618048922740876,1.7475270817585589) -- (0.5980989280301889,1.7754143277866787) -- (0.5258490542795644,1.8058746566866177) -- (0.4475642293094208,1.8376298982258752) -- (0.3907140453790507,1.859955797618285) -- (0.31846135993635016,1.8875203813962114) -- (0.24820098651428504,1.9135276000058812) -- (0.1995358288382622,1.9311155214303222) -- (0.1366130541984036,1.9533753887946328) -- (0.0902734224987697,1.9694419620401562) -- (0.0485223542657793,1.9836929240793175) -- (0,2) -- cycle;
\draw[line width=1.2pt,domain=1:1.5,smooth,variable=\y,color=red] plot ({\y*(3-2*\y)},{\y});
\draw[line width=1.2pt,domain=1:3.0,smooth,variable=\y,color=red] plot ({0.5*\y*(3-\y)},{\y});
\draw[-triangle 45, line width=1.0pt,scale=1] (0,0) -- (2.5,0) node[below] {$b$};
\draw[line width=1.0pt,-triangle 45] (0,0) -- (-0.2,0);
\draw[-triangle 45, line width=1.0pt,scale=1] (0,0) -- (0,3.3) node[left] {$a$};
\draw[line width=1.2pt,domain=1.5:3.0,smooth,variable=\y,color=red] plot (0,{\y});
\draw[line width=1.0pt,-triangle 45] (0,0) -- (0.0,-0.2);
\draw[scale=1] (1,0) node[below] {\footnotesize $1 $};
\draw[scale=1] (2,0) node[below] {\footnotesize $2 $};
\draw[scale=1] (0,1.0) node[left] {\footnotesize $1 $};
\draw[scale=1] (0,2.0) node[left] {\footnotesize $2 $};
\draw[scale=1] (0,3.0) node[left] {\footnotesize $3 $};
\draw[scale=1] (0.6,-0.3) node[below] {\footnotesize (i) $r=2,\; a=c_1+c_2+1>1.$};
\end{tikzpicture}
\end{center}
\end{minipage}%
\caption{The red boundary in both figures represent the stability region obtained by the necessary and sufficient conditions. The shaded regions represents the regions obtained by the first two steps of our algorithm as obtained in Proposition \ref{Pr-MathBiologyExample} }\label{Fig-MathBiology}
\end{figure}

\section{Stratifying the coefficient locus}\label{stratification}

This section complements \S\ref{algebraictheory}.
We explain that our algorithm produces an increasing filtration of the coefficients locus $\mathcal C_n(\bbd)$ by semialgebraic sets of increasing degrees (i.e. akin of a Taylor series). 
The starting point are Figures \ref{Fig-TD3}-\ref{Fig-ForFun} which nicely tell the story in degree two. The stability (or coefficient) locus $\mathcal C_n(\bbd)$ is the entire triangular region in Fig. \ref{Fig-TD}, and it is being covered incrementally by
semialgebraic sets whose defining polynomials are increasing in degree, starting with linear, then quadratic, etc, as for a ``Taylor series''. This can be formalized as follows.

Let $X\subset\bbr^N$ be a semialgebraic set. An \textit{increasing filtration} in $X$ is a sequence of nested semialgebraic subsets
\begin{equation}\label{filtration}
F_0\subset F_1\subset\cdots\subset F_k\subset\cdots \subset X.
\end{equation}
The filtration may be infinite, so we write $F_\infty$  its direct limit. The $F_i$'s form a filtration ``of'' $X$ if $F_\infty =X$, in which case $X = \bigsqcup X_i$ (set-theoretic disjoint union) where $X_i :=F_i\setminus F_{i-1}$. 

Recall that a semialgebraic set $X$ is given by polynomial equations and inequalities \eqref{semialgebraicset}. If the polynomial degrees are all $\leq k$, then we say that $X$ is defined by polynomials of degree $\leq k$. Note that one must be careful in associating a degree to $X$ since different polynomials, with different degrees, can define the same set, like $\{\alpha\in\bbr, \alpha^2<1\}=\{\alpha\in\bbr\;:\; -1 < \alpha < 1\}$.  In all cases, we can make sense of the following definition.

\begin{definition}
A semialgebraic ``Taylor'' filtration is a filtration as in the nested sequence \eqref{filtration}, where each $F_k$ is semialgebraic, defined by polynomials having degrees at most $k$.
\end{definition}

We now explain how our iterative process produces an increasing Taylor filtration for $\mathcal C_n(B)$. Start with
$$F_0=S_0 = \{(a_0,\ldots, a_{n-1})\in\bbr^n \;:\; \|p\|_1 < 2\}.$$
Then $F_0\subset\mathcal C_n(\bbd)$ is the first stratum in the filtration. Next set
$$S_1 =  \left\{ (a_0,\ldots, a_{n-1})\ : \
\sum |b_i|<1\right\},$$
where $b_i$'s are the coefficients of the first iterate system, as in \eqref{Eq-higherorder2}. We can now define $F_1 = S_0\cup S_1$. A key observation is that the inequalities defining $F_1$ are quadratic on the coefficients of the polynomial $a_0,\ldots, a_{n-1}$. Passing to the next iterate, we end up with a set of conditions which are now given by polynomials of degree three on the coefficients. This process continues, and we construct this way a sequence of subspaces $S_i$, and a filtration term
$$F_k = S_0\cup\cdots \cup S_k.$$
This filtration is semialgebraic by construction, and we have the series of inclusions $F_0\subseteq\cdots\subseteq F_k\subseteq\cdots$, whose linear stratum $F_0$ is given by all those polynomials with $\ell_1$-norm less than $2$. 
The convergence of our filtration to $\mathcal C_n(\bbd)$ is still open for more exploration, that is, if $F_\infty =\mathcal C_n(\bbd)$, or if it converges to something smaller. Finally, observe that in the case $n=2,$ we have 
\begin{eqnarray*}
S_0 &=& \{(\alpha,\beta)\ |\ |\alpha | + |\beta | < 1\}\\
S_1 &=& \{(\alpha,\beta)\ |\ 
|\alpha^2-\beta|+|\alpha\beta)|<1\}\\
S_2 & =& \{(\alpha,\beta)\ |\ 
|\alpha(\alpha^2-2\beta)|+|\beta(\alpha^2-\beta)|<1\}.
\end{eqnarray*}

\section{Conclusion}
An essential issue in the realm of dynamical systems is the analysis of the local stability of equilibrium solutions.  This involves analyzing the Jacobian matrix's spectrum or its characteristic polynomial roots. Local stability in a discrete-time dynamical system is achieved when the roots of the characteristic polynomial are clustered inside the open unit disk of the complex plane. From an algebraic perspective, the set of coefficients of real monic polynomials whose roots belong to the unit disk is a semialgebraic set; however, the set of equations and inequalities defining the semialgebraic set is neither unique nor easy to find.  The Jury's stability algorithm is widely acknowledged as a sophisticated algorithm for tackling the Schur stability problem. However, the cumbersomeness of the Jury's approach becomes apparent when dealing with polynomials of high degree or polynomials with coefficients that depend on parameters.   We use the authors' recent global stability results \cite{Al-Ca-Ka2023} for functions with mixed monotonicity to find sufficient conditions for the roots to cluster inside the unit disk. Proposition \ref{Pr-MainTheorem} and Theorem \ref{Th-Algorithm} can be applied successively to obtain an algorithm. The first step of the algorithm leads to the sufficient condition that a monic-polynomial is Schur-stable if its $\ell_1$-norm is less than two. Although this condition has been known since 1942, it is surprising that no algorithm has been devised that make use of this condition iteratively. Furthermore, by examining the signs of the coefficients in the polynomial, one can decipher the conditions obtained by our technique and devise an effective method. Finally, the developed technique has strengths, as illustrated in examples \ref{Ex-1} and \ref{Ex-2}, but limitations when many algorithm steps are needed. This issue is receptive to further investigation in our forthcoming research.

\vskip 1cm

\noindent{\textbf{Acknowledgement:}} The authors express their gratitude to the reviewers for their insightful remarks, which significantly enhanced the quality of our paper. Additionally, the first author expresses gratitude to Laura Gardini for the thought-provoking conversation held during the ECIT 2024 conference.  The first author is supported by sabbatical leave from the American University of Sharjah and by Maria Zambrano grant for attracting international talent from the Polytechnic University of Cartagena.
\bibliographystyle{unsrt}

\bibliography{AlSharawi-bibliography}

\end{document}